\documentclass[12pt]{article}
\usepackage{geometry}                % See geometry.pdf to learn the layout options. There are lots.
\usepackage{graphicx,color,mathtools}
\usepackage{amssymb,amsmath,amsthm,mathrsfs}
\usepackage[all,cmtip]{xy}
\usepackage{epstopdf, comment, url}
\usepackage{bm} % Bold math symbols
\usepackage{enumitem}

\newtheorem{theorem}{Theorem}[section]
\newtheorem{corollary}[theorem]{Corollary}
\newtheorem{lemma}[theorem]{Lemma}

\theoremstyle{remark}
\newtheorem*{remark}{Remark}

\numberwithin{equation}{section}

\DeclareMathOperator{\supp}{supp }

\DeclareMathOperator{\const}{const}

\DeclareMathOperator{\aut}{Aut}

\DeclareMathOperator{\dist}{dist}

\DeclareMathOperator{\im}{Im }

\DeclareMathOperator{\dyadic}{dyadic }
\DeclareMathOperator{\heavy}{heavy }
\DeclareMathOperator{\light}{light }

\DeclareMathOperator{\saw}{saw}
\DeclareMathOperator{\round}{round}

\DeclareMathOperator{\interior}{Int}
\DeclareMathOperator{\BC}{\mathcal{BC}}

\title{Beurling-Carleson sets, inner functions \\ and a semi-linear equation}
\author{Oleg Ivrii and Artur Nicolau}
\date{August 30, 2023}

\begin{document}

\maketitle

\abstract{Beurling-Carleson sets have appeared in a number of areas of complex analysis such as boundary zero sets of analytic functions, inner functions with derivative in the Nevanlinna class, cyclicity in weighted Bergman spaces, Fuchsian groups of Widom-type and the corona problem in quotient Banach algebras. After surveying these developments, we give a general definition of Beurling-Carleson sets and discuss some of their basic properties. We show that the Roberts decomposition characterizes measures that do not charge Beurling-Carleson sets.

\enlargethispage{4.2pt}

For a positive singular measure $\mu$ on the unit circle, let $S_\mu$ denote the singular inner function with singular measure $\mu$.
In the second part of the paper, we use a corona-type decomposition to relate a number of properties of singular measures on the unit circle such as  membership of $S'_\mu$ in the Nevanlinna class $\mathcal N$, area conditions on level sets of $S_\mu$ and wepability.
It was known that each of these properties holds for measures concentrated on Beurling-Carleson sets. We show that each of these properties implies that $\mu$ lives on a countable union of Beurling-Carleson sets. We also describe partial relations involving the membership of $S'_\mu$ in the Hardy space $H^p$, membership of $S_\mu$ in the Besov space $B^p$ and $(1-p)$-Beurling-Carleson sets and  give a number of examples which show that our results are optimal.

Finally, we show that measures that live on countable unions of  $\alpha$-Beurling-Carleson sets are almost in bijection with nearly-maximal solutions of $\Delta u = u^p \cdot \chi_{u > 0}$ when $p > 3$ and $\alpha =  \frac{p-3}{p-1}$.
}

\section{Introduction}

A {\em Beurling-Carleson set} $E$ is a closed subset of the unit circle $\partial \mathbb{D}$  of zero length whose complementary arcs $\{ J \}$ satisfy
\begin{equation}
\label{eq:original-definition}
\| E \|_{\BC} = \sum_J |J| \log \frac{1}{|J|} < \infty.
\end{equation}
Beurling-Carleson sets were introduced by A.~Beurling \cite{beurling}, who showed that they constitute boundary zero sets of holomorphic functions on the unit disk that are H\"older continuous up to the boundary.
Several years later, L.~Carleson \cite{carleson} constructed outer functions that vanished to arbitrarily order on $E$. This construction was later improved to infinite order by Taylor and Williams \cite{taylor-williams}. Since then, Beurling-Carleson sets appeared in a number of areas of complex analysis such as inner functions, weighted Bergman spaces, Fuchsian groups and the corona problem.

In this paper, we will also consider Beurling-Carleson sets with respect to other gauge functions, although we will be mainly interested in usual Beurling-Carleson sets and $\alpha$-Beurling-Carleson sets with $0 < \alpha < 1$. These are defined by the condition
\begin{equation}
\label{eq:alpha-definition}
\| E \|_{\BC_\alpha} = \sum_J |J|^\alpha < \infty,
\end{equation}
in place of (\ref{eq:original-definition}).

\subsection{Derivative in Nevanlinna class}

An inner function is a bounded analytic function on the unit disk $\mathbb{D}$ which has unimodular radial limits almost everywhere on $\partial \mathbb{D}$.
Beurling-Carleson sets play an important role in understanding inner functions with derivative in the Nevanlinna class $\mathcal N$, which consists of analytic functions $f(z)$ on the unit disk for which
$$
\lim_{r \to 1} \int_{|z|=r} \log^+|f(z)| < \infty.
$$
Suppose $\mu$ is a positive singular measure on the unit circle and
$$
S_\mu(z) = \exp \biggl ( - \int_{\partial \mathbb{D}} \frac{\zeta+z}{\zeta-z} d\mu(\zeta) \biggr ), \qquad 
|z| < 1,
$$
 is the associated singular inner function. 
On the unit circle, the radial boundary values of $|S'_\mu|$ are given by
$$
|S'_\mu(z)| =  2 \int_{\partial \mathbb{D}} \frac{|d\zeta|}{|\zeta -  z |^2}, \qquad |z|=1,
$$
which could be infinite.
 M.~Cullen \cite{cullen} observed that if $\mu$ is  concentrated on a Beurling-Carleson set, then $S'_\mu \in \mathcal N$. The converse does not hold in general: there are singular inner functions $S_\mu$ with $S'_\mu \in \mathcal N$ for which the support of $\mu$ is not contained in a single Beurling-Carleson set. One consequence of \cite{ivr19} is that the condition $S'_\mu \in \mathcal N$ implies that $\mu$ lives on a countable union of Beurling-Carleson sets. The original proof used the classification of nearly-maximal solutions of the Gauss curvature equation $\Delta u = e^{2u}$. In Section \ref{sec:corona-construction}, we will give an elementary proof of this fact using a corona-type decomposition.

\begin{theorem}
\label{N-theory}
Let $\mu \ge 0$ be a singular measure on $\partial \mathbb{D}$. Consider the following conditions:
\begin{enumerate}
\item[{\em (0)}]  The measure $\mu$ is supported on a Beurling-Carleson set.
\item[{\em (1)}]  $S'_{\mu} \in  \mathcal N$.
\item[{\em (2)}] $S_\mu$ satisfies the area condition: for every $0 < c < 1$,
 \begin{equation}
\label{eq:area-condition}
\int_{\{z \in \mathbb{D}:\, |S_\mu(z)| < c\}} \frac{dA(z)}{1-|z|} < \infty.
\end{equation}
\item[{\em (3)}] The measure $\mu$ is concentrated on a countable union of Beurling-Carleson sets.
\end{enumerate}
We have $(0) \Rightarrow (1) \Rightarrow (2) \Rightarrow (3)$.
\end{theorem}

\subsection{Quotient Banach algebras}

Another important perspective on Beurling-Carleson sets stems from the work \cite{GMN} of P.~Gorkin, R.~Mortini and N.~Nikolskii, who studied the corona problem in the quotient space $H^\infty/IH^\infty$, where $I$ is an inner function. They noticed that point evaluations at the zeros of $I$ are dense in the maximal ideal space $\mathfrak M$ of $H^\infty/IH^\infty$ if and only if there exists a $ 0<c<1$ for which the sub-level set
$$
\Omega_c = \{ z \in \mathbb{D} : |I(z)| < c \}
$$
is contained within a bounded hyperbolic distance of the zero set of $I$. In this case, one says that $I$ has the {\em weak embedding property}\/. In \cite{borichev}, A.~Borichev introduced the class of {\em wepable} inner functions, i.e.~inner functions that could be made WEP if multiplied by a suitable Blaschke product. Consider the condition

{\em
\begin{enumerate}
\item[$(1')$]  $S_\mu$ is wepable.
\end{enumerate}
}

In \cite{BNT}, the authors proved that  $(0) \Rightarrow (1') \Rightarrow (2)$. Together with the implication $(2) \Rightarrow (3)$ from Theorem \ref{N-theory}, this shows that up to countable unions, the collection of measures $\mu$ for which $S_\mu$ is wepable also coincides with measures that are concentrated on Beurling-Carleson sets. 

\begin{remark}
Taking countable unions is necessary since there exist atomic measures $\mu$ for which $S_\mu$ is not wepable. See the proof of \cite[Theorem 3]{BNT}.
\end{remark}

\subsection{Derivative in $H^p$}

Next, we use a corona-type decomposition to study singular inner functions with derivative in the Hardy space $H^p$. We stick to the range of exponents $0 < p < 1/2$, since derivatives of singular inner functions are never in $H^{1/2}$.

\begin{theorem}
\label{Hp-theory}
Suppose $0 < p < 1/2$ and $\mu \ge 0$ is a singular measure on $\partial \mathbb{D}$. Consider the following conditions:
\begin{enumerate}
\item[{\em (1)}]  $S'_{\mu} \in  H^p$.
\item[{\em (2)}] $S_\mu$ satisfies the $(1+p)$-area condition: for every $0 < c < 1$,
 \begin{equation}
\label{eq:area-condition2}
\int_{\{z \in \mathbb{D}:\, |S_\mu(z)| < c\}} \frac{dA(z)}{(1-|z|)^{1+p}} < \infty.
\end{equation}
\item[{\em (3)}] The measure $\mu$ is concentrated on a countable union of $(1-p)$-Beurling-Carleson sets.
\end{enumerate}
We have $(1) \Rightarrow (2) \Rightarrow (3)$.
\end{theorem}

Unfortunately, it is no longer true that if $\mu$ is supported on a  $(1-p)$-Beurling-Carleson set, then $S'_{\mu} \in  H^p$.

We say that a finite measure $\mu \ge 0$ satisfies a property {\em up to countable sums} if it can be written
as a countable sum of finite measures $\mu_k \ge 0$ satisfying the property.
In Section \ref{sec:derivative-hardy}, we will see that  conditions (1) and (3) are different even after allowing countable sums. Nevertheless, in Section \ref{sec:besov-spaces}, we will show that conditions (1) and (2) agree after passing to countable sums.

We mention an additional condition on the measure $\mu$, equivalent to (2), due to P.~Ahern \cite{ahern} and A.~Reijonen and T.~Sugawa \cite{reijonen-sugawa}:

{\em
\begin{enumerate}
\item[$(2')$] The integral
$$
\int_{\mathbb{D}} |S_\mu'(z)|^q (1-|z|^2)^{-p+(q-1)} dA(z) < \infty,
$$
for some (and hence, all) $1 \le q \le 2$.
\end{enumerate}
}
When $q = 1$, the above condition says that $S'_\mu$ belongs to the Besov space $B^p$. The implication $(1) \Rightarrow (2')$ can also be found in Ahern's paper.

%For instance, we will show that if $\mu$ lives on an $(1-q)$-Beurling-Carleson set with $q > \frac{p}{1-p}$, then $S'_\mu \in H^p$. We will provide two examples showing that this exponent is sharp.

\subsection{Differential equations}

  It was observed   in \cite{ivr19} that characterizing inner functions with derivative in Nevanlinna class amounts to understanding nearly-maximal solutions of the Gauss curvature equation $\Delta u = e^{2u}$. These turn out to be in one-to-one correspondence with measures that live on countable unions of Beurling-Carleson sets. We refer the reader to Section \ref{sec:background-PDE} for the relevant definitions and background on semi-linear equations.
      
  In Section \ref{sec:nearly-maximal}, we show the following theorem which partially characterizes the nearly-maximal solutions of $\Delta u = u^p \cdot \chi_{u > 0}$\,:% are almost in a one-to-one correspondence with measures that live on countable unions of $\alpha$-Beurling-Carleson sets with $\alpha = \frac{p-3}{p-1}$\,:

\begin{theorem}
\label{nearly-maximal}
{\em (i)} When $p > 3$, deficiency measures of nearly-maximal solutions are concentrated on countable unions of $\alpha$-Beurling-Carleson sets, where $\alpha = \frac{p-3}{p-1}$.
Conversely, any finite positive measure on the unit circle concentrated on a  countable union of $\beta$-Beurling-Carleson sets for some $\beta < \alpha$ arises as the deficiency measure of some nearly-maximal solution.
 
{\em (ii)} When $1 < p \le 3$, the only nearly-maximal solution is the maximal one.
\end{theorem}

It is natural to wonder if there is a precise correspondence between nearly-maximal solutions of $\Delta u = u^p \cdot \chi_{u > 0}$ and measures that live on countable unions of $\alpha$-Beurling-Carleson sets. Unfortunately, with our current techniques, we are unable to either prove or disprove this tantalizing hypothesis.

\section{Notes and references}

\subsection{Weighted Bergman spaces}

 Beurling-Carleson sets also arise naturally  in the study of cyclic functions in the weighted Bergman spaces $A_\alpha^p$, which consists of holomorphic functions on the unit disk satisfying 
$$
\| f \|_{A^p_\alpha}^p \, = \, \int_{\mathbb{D}} |f(z)|^p \, (1-|z|)^\alpha dA(z) \, < \, \infty, \qquad \alpha > -1, \quad 1 < p < \infty.
$$
A function $f \in A^p_\alpha$ is {\em cyclic} if the closure of the set $\{p f : p \text{ polynomial}\}$ is dense in $A_\alpha^p$. One question that puzzled mathematicians in the late 1960s was: {\em When is the singular inner function $S_\mu$ cyclic?}
It was not difficult to show that if $\mu$ is concentrated on a Beurling-Carleson set, then the singular inner function $S_\mu$ could not be cyclic.  In the other direction, it was known that if $\mu$ had modulus of continuity bounded by $C t \log(1/t)$, then $S_\mu$ was cyclic. The gap between Beurling-Carleson sets and the $t \log 1/t$ condition stood for a number of years until it was resolved independently by B.~Korenblum \cite{korenblum-cyc} and J.~Roberts \cite{roberts}. Roberts' approach used an elegant structure theorem for measures that do charge Beurling-Carleson sets. In Section \ref{sec:BC-sets}, we will prove a converse of Roberts' result, thereby giving a description of positive singular measures that do not charge Beurling-Carleson sets. %Korenblum's approach involves a clever linear programming argument.

\subsection{Model spaces}

Let $A^\infty$ denote the space of holomorphic functions on the open unit disk which extend to smooth functions on the closed unit disk.
To an inner function $F(z)$, one can associate the {\em model space} $K_F = H^2 \ominus FH^2$. K.~Dyakonov and D.~Khavinson \cite{DK} were curious as to whether $K_F$ contained smooth functions. They showed that $K_F \cap A^\infty = \{ 0 \}$ if and only if $F = S_\mu$ where $\mu$ does not charge Beurling-Carleson sets.

In a recent work, A.~Limani and B.~Malman \cite{LM22} asked the opposite question: when is $K_F \cap A^\infty$ dense in
$K_F$? They showed that this occurs if and only if $F = BS_\mu$, where $B$ is an arbitrary Blaschke product and $\mu$ is concentrated on a countable union of Beurling-Carleson sets.

\subsection{Character-automorphic functions} Widom \cite{widom} and Pommerenke \cite{pommerenke, pommerenke2} studied functions which were character-automorphic under Fuchsian groups of convergence type. A {\em character} $v$ of a Fuchsian group $\Gamma \subset \aut(\mathbb{D})$ is a homomorphism of $\Gamma$ to the unit circle. A function $f$ on the unit disk is called {\em character automorphic} if 
$$
f(\gamma(z)) = v(\gamma) \cdot f(z), \qquad \gamma \in \Gamma.
$$
 One natural character automorphic function is the Blaschke product $g(z)$ whose zeros constitute
 an orbit of $\Gamma$ (it is related to the Green's function of $\mathbb{D}/\Gamma$). If $g(z)$ has zeros at the points $\{ \gamma(0) : \gamma \in \Gamma \}$, i.e.~$$g(z) = \prod_{\gamma \in \Gamma} -\frac{\overline{\gamma(0)}}{|\gamma(0)|} \cdot \frac{z-\gamma(0)}{1-\overline{\gamma(0)}z},$$ then
 $$
 |g'(z)| = \sum_{\gamma \in \Gamma} |\gamma'(z)|, \qquad |z| = 1.
 $$
For a character $v$, let $H^\infty(\Gamma, v)$ denote the space of bounded holomorphic $v$-automorphic functions. Building on the work of Widom, Pommerenke \cite{pommerenke} showed that
$$
g' \in \mathcal N \quad \Longleftrightarrow \quad H^\infty(\Gamma, v) \ne \{ \const \}, \text{ for every }v,
$$
and observed that the above condition is satisfied if the limit set $\Lambda(\Gamma)$ is a Beurling-Carleson set. %Here, we show that if $\Lambda(\Gamma)$ is contained in a countable union of Beurling-Carleson sets, then the converse holds, i.e.~$g' \in \mathcal N$.

Pommerenke \cite[Theorem 2]{pommerenke2}  also showed that $\Lambda$ is a Beurling-Carleson set if and only if there is a $\Gamma$-invariant holomorphic vector field $h(z) \frac{\partial}{\partial z}$ on the unit disk with $h'(z) \in H^\infty$.
  
\subsection{Fat Beurling-Carleson sets}

A related class of sets was introduced by S.~Khruschev, which is natural to call {\em fat Beurling-Carleson sets}\/. These are closed subsets of the unit circle which satisfy the entropy condition (\ref{eq:original-definition}), but have positive Lebesgue measure. Amongst other things, Khruschev showed that if $K$ is a closed subset of the unit circle which does not contain any fat Beurling-Carleson sets, then there is a sequence of polynomials $p_n(z)$ which tend to 1 in the Bergman space $A^2(\mathbb{D})$ but to $0$ in $C(K)$. Conversely, if such a sequence of polynomials exists, then $K$ cannot contain any fat Beurling-Carleson sets.

The proof presented in  \cite[Chapter II.3]{HJ} uses a structure theorem due to N.~G.~Makarov \cite{makarov}. Given a closed subset $K$ of the circle which does not contain fat Beurling-Carleson sets and an arc $I \subset \partial \mathbb{D}$, there exists a measure $\mu = \mu_I$ supported on $I \setminus K$ which satisfies 

(i) $\mu(I) \ge |I| \log \frac{1}{|I|}$,

(ii) $\mu(J) \le 3 |J| \log \frac{1}{|J|}$ for any arc $J \subseteq I$.

\noindent The first condition implies that $\mu$ has substantial mass, while the second condition says that $\mu$ is spread out.

For more applications of fat Beurling-Carleson sets, we refer the reader to \cite{LM21b, LM21a, Mal22}.

\section{Beurling-Carleson sets}
\label{sec:BC-sets}

In this section, we give a general definition of Beurling-Carleson sets and discuss some of their basic properties. We say that $\phi: [0, 1] \to [0,\infty)$ is a {\em regular gauge function} if
\smallskip

\begin{enumerate}[leftmargin=17mm]
\item[(G1)] One can write $$\phi(t) \, = \, t \cdot \phi_1(t) \, = \, t\int_t^1 \frac{ds}{\lambda(s)},$$
where $\lambda(t)$ is a non-negative function such that $\int_0^1 \frac{ds}{\lambda(s)} = \infty$.
\item[(G2)] The function $\lambda(t)$ satisfies the  doubling condition
\begin{equation}
\label{eq:lambda-two}
 \lambda(\theta \cdot t) \asymp \lambda(t), \qquad \theta \in [1, 2].
\end{equation}
\item[(G3)] There exists a constant $C > 0$ such that
$$
\sum_{k=0}^\infty \phi(2^{-k} t) \le C\phi(t), \qquad t \in [0, 1].
$$
\end{enumerate}
A closed subset $E$ of the unit circle of zero length is called a {\em $\phi$-Beurling-Carleson set} if
\begin{equation}
\| E \|_{\BC_\phi} = \sum_k \phi(|J_k|) < \infty,  \end{equation}
where the sum is over the complementary arcs $\{J_k\}$ of $E$. 

For each $n \ge 0$, we can partition the unit circle into $2^n$ {\em dyadic arcs} of generation $n$:
$$
\bigl \{ z \in \partial \mathbb{D} \, : \, k \cdot 2^{-n} \cdot 2\pi \, < \, \arg z \, < \,  (k+1) \cdot 2^{-n}  \cdot 2\pi \bigr \}
, \qquad k = 0, 1, \dots, 2^n-1.  $$
  We  denote the collection of dyadic arcs of generation $n$ by $\mathcal D_n$. The {\em dyadic grid} $\mathcal D = \bigcup_{n=0}^\infty \mathcal D_n$ is the collection of all dyadic arcs.
 
Given a closed set $E$, the {\em Privalov star} $K_E$ is defined as the union of the Stolz angles of opening $\pi/2$ emanating from points of $E$. 

The following lemma provides several other characterizations of Beurling-Carleson sets:

\begin{lemma}
\label{dyadic-interval-sums}
Let $E$ be a closed subset of the unit circle of zero length. Denote the complementary arcs 
by $\{J_k\}$, i.e.~$\partial \mathbb{D} \setminus E = \bigcup J_k$. If $\phi$ is a regular gauge function, then the following quantities are comparable:

{\renewcommand{\arraystretch}{1.5}
{\em 
\begin{tabular}{ c l l }
 (a) & Arc sum: & $\sum_k \phi(|J_k|)$ \\ 
 (b) & Distance integral: & $\int_{\partial \mathbb{D} \setminus E}  \phi_1(\dist(x,E)) \, dx$ \\  
 (c) & Dyadic arc sum: & $\sum_{\substack{I \dyadic \\   I \cap E \ne \emptyset}}  \frac{|I|^2}{\lambda(|I|)}$ \\    
 (d) & Privalov star integral: & $\int_{K_E} \frac{dA(z)}{\lambda(1-|z|)}$
\end{tabular}
}}
\end{lemma}

\smallskip

\begin{remark}
In (d), instead of integrating over the Privalov star $K_E$, one can also integrate over the region
$$
\Omega_E = \mathbb{D} \setminus \bigcup_k Q_{J_k},$$ where 
$$
Q_J \, = \, \biggl \{z \in \mathbb{D} \, : \, \frac{z}{|z|} \in J, \ 0 < 1-|z| < |J| \biggr \}
$$ is the Carleson box with base $J \subset \partial \mathbb{D}$.
Alternatively, one can integrate over the domain
$$
\Omega^{\dyadic}_E = \bigcup_{\substack{I \dyadic \\   I \cap E \ne \emptyset}} T_I,
$$
where
$$
T_I \, = \, \biggl \{z \in \mathbb{D} \, : \, \frac{z}{|z|} \in I, \ \frac{|I|}{2} < 1-|z| < |I| \biggr \}
$$ 
denotes the top half of the Carleson box which rests on $I$.
\end{remark}

\subsubsection*{Examples}

\begin{enumerate} 
\item[(i)] If $\phi(t) = t \log t^{-1}$, then $\lambda(t) = t$ and we recover the usual Beurling-Carleson condition:
$$
\sum_k |J_k| \log \frac{1}{|J_k|} \ \asymp \ \int_{[0,1]\setminus E} \log \frac{1}{\dist(x,E)} dx
\ \asymp \ \sum_{\substack{I \dyadic \\   I \cap E \ne \emptyset}} |I| 
\ \asymp \ \int_{K_E} \frac{dA(z)}{1-|z|}.
$$

\item[(ii)] If $\phi(t) = t^\alpha$ with $0<\alpha<1$, then $\lambda(t) \sim \frac{t^{2 - \alpha}}{1-\alpha}$ as $t \to 0^+$ and we get the $\alpha$-Beurling-Carleson condition:
$$
\sum_k |J_k|^\alpha \ \asymp \ \int_{[0,1]\setminus E}\dist(x,E)^{\alpha-1} dx
\ \asymp \ \sum_{\substack{I \dyadic \\   I \cap E \ne \emptyset}} |I|^\alpha
\ \asymp \ \int_{K_E} \frac{dA(z)}{(1-|z|)^{2-\alpha}}.
$$
\end{enumerate}

\begin{proof}[Proof of Lemma \ref{dyadic-interval-sums}]
The comparability of the ``arc sum'' and the ``distance integral'' follows after subdividing each complementary interval $J_k$ into Whitney arcs and applying the estimate (G3), while the comparability of the ``distance integral'' and the ``Privalov star integral'' follows from integrating in polar coordinates.

It remains to relate the ``Privalov star integral'' and the ``dyadic arc sum.''  By the doubling property (G2) of $\lambda$, we have
$$
\int_{T_I} \frac{dA(z)}{\lambda(1-|z|)} \asymp  \frac{|I|^2}{\lambda(|I|)}.
$$ 
Summing over the dyadic arcs $I$ which meet $E$ gives
$$
\int_{\Omega^{\dyadic}_E} \frac{dA(z)}{\lambda(1-|z|)} \, \asymp \, 
\sum_{\substack{I \dyadic \\   I \cap E \ne \emptyset}}  \frac{|I|^2}{\lambda(|I|)}.
$$
Inspection shows that
\begin{equation}
\label{eq:dyadic-goal}
\int_{\Omega^{\dyadic}_E} \frac{dA(z)}{\lambda(1-|z|)} \, \asymp \,
\int_{\Omega_E} \frac{dA(z)}{\lambda(1-|z|)} \, \asymp \, 
 \int_{K_E} \frac{dA(z)}{\lambda(1-|z|)}.
\end{equation}
The proof is complete.
\end{proof}

%%
\begin{comment}
Summing
$$
\int_{\Omega_E \,\cap\, \{\arg z \in J_k \}} \frac{dA(z)}{\lambda(1-|z|)} \, = \,
 |J_k| \int_{|J_k|}^1 \frac{dt}{\lambda(t)}  \, \asymp \, 
  \phi(|J_k|)
$$
over $k$ shows that
$$
\int_{\Omega_E} \frac{dA(z)}{\lambda(1-|z|)} \, \asymp \, 
 \sum_k \phi(|J_k|).
$$
\end{comment}
%%

\subsection{Dyadic grid with respect to a gauge function}
 
 A {\em $\phi$-dyadic grid}
is a collection of dyadic arcs $\mathcal D_\phi = \bigcup_j \mathcal D_{n_j}$ where the sequence $\{n_j\}$ satisfies
\begin{equation}
\label{eq:scaling-condition}
\int^{2^{-n_j}}_{2^{-n_{j+1}}} \frac{dt}{\lambda(t)} \, \asymp \, \int_{2^{-n_j}}^{1} \frac{dt}{\lambda(t)} \, \asymp \, \phi_1(2^{-n_j}) , \qquad j=1,2,\dots
\end{equation}
In particular, the above condition implies that $\phi_1(|I|) \asymp \phi_1(|J|)$ whenever $I \in \mathcal D_{n_{j+1}}$ and $J \in \mathcal D_{n_{j}}$.

\subsubsection*{Examples}

\begin{enumerate} 
\item[(i)]  If $\phi(t) = t\log t^{-1}$, one can take $n_j = 2^j$ and obtain the super-dyadic scales $2^{-n_j} = 2^{-2^j}$. In this case, $\lambda(t) = t$.

\item[(ii)] When $\phi(t) = t^\alpha$, $\alpha > 0$, one can take $n_j = j$ and get the standard dyadic scales $2^{-j}$.  In this case,  $\lambda(t) \asymp \frac{t^{2 - \alpha}}{1-\alpha}$ as $t \to 0$.
\end{enumerate}

\subsubsection*{Dyadic shells and boxes}
 We can decompose the unit disk $\mathbb{D}$ into $\phi$-dyadic shells:
$$
\mathcal A_{\phi, 0} = \{z \in \mathbb{D} : |z| < 1- 2^{-n_1}\},
$$
and
$$
\mathcal A_{\phi, j} = \{z \in \mathbb{D} : 1 - 2^{-n_j} < |z| < 1- 2^{-n_{j+1}}\}, \quad j = 1, 2, \dots
$$
Each shell can be further subdivided into $\phi$-dyadic boxes:
$$
T_I^\phi = \mathcal A_{\phi, j} \cap Q(I) = \bigl  \{ re^{i\theta} \in \mathbb{D}: \ \theta \in I, \  1 - 2^{-n_j} < r < 1- 2^{-n_{j+1}} \bigr \},
$$
where $I$ ranges over $\mathcal D_{n_j}$. For further reference, we note that
\begin{equation}
\label{eq:top-phi-integral}
\int_{T_I^\phi}  \frac{dA(z)}{\lambda(1-|z|)} \, \asymp \, |I| \cdot \phi_1(|I|) \, = \, \phi(|I|).
\end{equation}

\subsection{Roberts decomposition}

In a remarkable work \cite{roberts}, J.~Roberts came up with an elegant structure theorem for measures that do not charge Beurling-Carleson sets. This is done by {\em grating} a measure with respect to finer and finer partitions associated to a $\phi$-dyadic grid.

\begin{theorem}
\label{roberts-decomposition}
Let $\phi:[0,1]\to [0,\infty)$ be a regular gauge function and  $\mathcal D_\phi = \bigcup \mathcal D_{n_k}$ be a $\phi$-dyadic grid. Let $\mu$ be a finite positive measure on $\partial \mathbb{D}$. Then for any integer $j_0 \ge 0$ and $C > 0$, one can decompose $\mu = \sum_{j=1}^\infty \mu_j + \mu_\infty$ so that $\mu_j(I) \le C \phi(|I|)$ for any $I \in \mathcal D_{n_{j + j_0}}$ and $\mu_\infty$ is concentrated on a $\phi$-Beurling-Carleson set.
\end{theorem}

\begin{proof} For each $j = 1, 2, \dots$, we can define a partition $P_j$ of the unit circle into $2^{n_{j + j_0}}$ arcs of equal length (we consider half-open arcs which contain only one of the endpoints, for example, the left endpoint). Since $2^{n_{j + j_0}}$ divides $2^{n_{j+ j_0+1}}$, each next partition can be chosen to be a refinement of the previous one.

To define $\mu_1$, consider the arcs in the partition $P_1$. Call an arc $I \in P_1$  {\em light} if
$\mu(I) \le C\phi(|I|)$ and {\em heavy} otherwise. On a light arc, take $\mu_1 = \mu$, while
on a heavy arc, let $\mu_1$ be a  multiple of $\mu$ so that the mass $\mu_1(I) = C\phi(|I|)$. The measure $\mu_1$ will be called the {\em grated measure} of $\mu$ with respect to the partition $P_1$. Clearly, $\mu_1 \le \mu$.
Consider the difference  $\mu - \mu_1$ and grate it with respect to the partition $P_2$ to form the measure $\mu_2$, then consider  $\mu - \mu_1 - \mu_2$ and grate it with respect to $P_3$ to form $\mu_3$, and so on. Continuing in this way, we obtain a sequence of measures $\mu - \mu_1, \mu - \mu_1 - \mu_2, \dots$ where each next measure is supported on the heavy arcs of the previous generation.

By construction, the bound $\mu_j(I) \le C \phi(|I|)$, $I \in \mathcal D_{n_{j+j_0}}$ holds for all $j$, while the residual measure $\mu_\infty$ is supported on the set of points which always lie in heavy arcs.  A fortiori, the residual measure is supported on the complement of the light arcs and we need to show that
$
\sum_{I \light} \phi(|I|) < \infty.
$
The scaling condition (\ref{eq:scaling-condition}) tells us that
$$
\sum_{\substack{I \subset J \\ I \in \mathcal D_{n_{j+1}}}}  \phi(|I|) \, = \, |J| \cdot \phi_1(|I|) \, \le \,
 C \phi(|J|), 
 \qquad J \in \mathcal D_{n_j}.
$$
Since a light arc of generation $j \ge 2$ is contained in a heavy one, 
\begin{align*}
\sum_{\text{light}} \phi(|I|) & \lesssim
2^{n_{j_0}}\phi(2^{-n_{j_0}}) + \sum_{\text{heavy}} \phi(|J|) \\ 
& = 2^{n_{j_0}}\phi(2^{-n_{j_0}}) + \frac{1}{C} \sum_j \sum_{\substack{J \in \mathcal D_{n_{j + j_0}} \\  J \heavy}} \mu_j(J) \\
& \le 2^{n_{j_0}}\phi(2^{-n_{j_0}}) + \frac{1}{C} \cdot \mu(\partial \mathbb{D}).
\end{align*}
The proof is complete.
\end{proof}

\begin{corollary}
\label{roberts-series}
If $\mu$ does not charge $\phi$-Beurling-Carleson sets, then for any $j_0 \ge 0$ and $C > 0$, one can write $\mu = \sum \mu_j$ where $\mu_j(I) \le C \phi(|I|)$ for any $I \in \mathcal D_{n_{j+j_0}}$.
\end{corollary}

We now show the converse of Corollary \ref{roberts-series}:

\begin{corollary}
Suppose that there exists a constant $C > 0$ so that for any offset $j_0 \ge 0$, one can decompose the measure $\mu$ into a countable sum $\mu = \sum \mu_j$ so that $\mu_j(I) \le C \phi(|I|)$ for any $I \in \mathcal D_{n_{j + j_0}}$. Then $\mu$ does not charge $\phi$-Beurling-Carleson sets.
\end{corollary}

\begin{proof}
Let $E$ be a $\phi$-Beurling-Carleson set. By Lemma \ref{dyadic-interval-sums}, for any $\varepsilon > 0$, we can choose the offset $j_0 \ge 0$ sufficiently large so that 
$$
\sum_{j=1}^\infty  \int_{K_E \cap \mathcal A_{\phi, j+j_0}} \frac{dA(z)}{\lambda(1-|z|)}   < \varepsilon.
$$
In view of (\ref{eq:top-phi-integral}), we have
\begin{align*}
\mu_j(E) & = \sum_{\substack{I \in \mathcal D_{n_{j+j_0}} \\   I \cap E \ne \emptyset}} \mu_j(I) \\
& \le C \sum_{\substack{I \in \mathcal D_{n_{j+j_0}} \\   I \cap E \ne \emptyset}} \phi(|I|) \\
& \le  C' \int_{K_E \cap \mathcal A_{\phi, j+j_0}} \frac{dA(z)}{\lambda(1-|z|)}.
\end{align*}
Summing over $j = 1, 2, \dots$ yields $\mu(E) \le C' \varepsilon$. Since $\varepsilon > 0$ was arbitrary,   $\mu(E) = 0$ as desired.
\end{proof}

\subsection{Local behaviour}

The following theorem roughly says that measures on the unit circle which are sufficiently spread out cannot charge Beurling-Carleson sets:

\begin{theorem}
\label{diffuse-thm}
 Suppose $w(\varepsilon)/\varepsilon$ is strictly decreasing on $(0,1]$. Then,
$\mu(E) = 0$ for every $\phi$-Beurling-Carleson set $E$ and positive measure $\mu$ on the unit circle satisfying the modulus of continuity condition 
$$\mu(I) \le c \cdot w(|I|), \qquad I \subset \partial \mathbb{D},$$
 if and only if 
 \begin{equation}
 \int_0^1 \frac{\varepsilon}{\lambda(\varepsilon) w(\varepsilon)} d\varepsilon = \infty.
 \end{equation}
\end{theorem}

\smallskip

In full generality, Theorem \ref{diffuse-thm} was proved by R.~D.~Berman, L.~Brown and W.~S.~Cohn \cite[Corollary 4.1] {BBC}. For usual Beurling-Carleson sets, Theorem \ref{diffuse-thm} goes back to P.~Ahern \cite{ahern} and J.~H.~Shapiro \cite{shapiro}.
\subsubsection*{Examples}

\begin{enumerate} 
\item[(i)]  If $\phi(t) = t\log t^{-1}$, the above condition reads: $\int_0^1 w(\varepsilon)^{-1} d\varepsilon = \infty.$

\item[(ii)] For $\phi(t) = t^\alpha$, $\alpha > 0$, the condition becomes  $\int_0^1 \varepsilon^{\alpha-1} w(\varepsilon)^{-1} d\varepsilon = \infty.$
\end{enumerate}

\smallskip

\begin{theorem}
\label{local-behaviour}
Suppose $\mu$ is a measure on the unit circle supported on a countable union of $\phi$-Beurling-Carleson sets.  Let 
 $\mu(x,\varepsilon) = \mu(I(x,\varepsilon))$ where $I(x, \varepsilon)$ is the arc on the unit circle centered at $x$ of length $2\varepsilon$.  For almost every point $x$ on the unit circle with respect to $\mu$,
$$\int_0^1 \frac{\varepsilon}{\lambda(\varepsilon) \mu(x, \varepsilon)} d\varepsilon < \infty.$$
\end{theorem}

\begin{proof}
It suffices to consider the case when $\mu$ is supported on a single $\phi$-Beurling-Carleson set $E$.
 Since $\mu$ is a singular measure, for $\mu$-a.e.~$x \in \partial \mathbb{D}$, $\lim_{\varepsilon \to 0} \frac{\mu(x, \varepsilon)}{\varepsilon} = \infty$.
To prove the lemma, we will show that the double integral
$$
\int_{E} \int_0^1 \frac{\varepsilon}{\lambda(\varepsilon) \mu(x, \varepsilon)} \, d\varepsilon d\mu(x) \lesssim \| E \|_{\BC_\phi}.
$$
For a point $x \in \partial \mathbb{D}$, we write $S(x)$ for the Stolz angle of opening $\pi/2$ with vertex at $x$.
Recall that $K_E$ denotes the union of the Stolz angles emanating from points $x \in E$. According to Lemma \ref{dyadic-interval-sums},
$$
 \| E \|_{\BC_\phi} \asymp \int_{K_E} \frac{dA(z)}{\lambda(1-|z|)}.
$$
We subdivide the above integral over individual Stolz angles:
$$
\int_{K_E} \frac{dA(z)}{\lambda(1-|z|)} = \int_E \int_{S(\zeta)} \eta(z) \cdot \frac{dA(z)}{\lambda(1-|z|)} d\mu(\zeta),
$$
where the function $\eta(z) = \mu(I_z)^{-1}$ measures how many Stolz angles contain $z$. Here, 
 $I_z$ is the arc of the unit circle that consists of points $\zeta$ for which $z \in S(\zeta)$.
From
\begin{align*}
\int_{S(\zeta) \cap \{1-|z| = \varepsilon\}} \eta(z) \cdot \frac{|dz|}{\lambda(1-|z|)} & \ge \frac{\varepsilon}{\lambda(\varepsilon)} \cdot \min_{z \in S(\zeta) \cap \{1-|z| = \varepsilon\}} \mu(I_z)^{-1} \\
& \ge \frac{\varepsilon \cdot \mu(\zeta, 3\varepsilon)^{-1}}{\lambda(\varepsilon) },
\end{align*}
we deduce that
$$
\int_{E} \int_0^1 \frac{\varepsilon \cdot \mu(\zeta, 3\varepsilon)^{-1}}{\lambda(\varepsilon) } d\varepsilon d\mu(\zeta) \lesssim \| E \|_{\BC_\phi}
$$
as desired.
\end{proof}

 \begin{corollary}
\label{local-behaviour2}
Suppose $\mu$ is a measure on the unit circle supported on a countable union of $\phi$-Beurling-Carleson sets. For any $c  > 0$, the region
 $$\Omega_c = \{ z \in \mathbb{D} : P_\mu(z) > c\}$$ is ``thick'' at almost every point $x$ on the unit circle with respect to $\mu$, in the sense that
\begin{equation}
\label{eq:general-thickness}
\int_0^1 \frac{\eta(x, \varepsilon)}{\varepsilon \cdot \lambda(\varepsilon)} d\varepsilon < \infty,
\end{equation}
where $\eta(x, \varepsilon) = \pi \varepsilon - |\partial B(x, \varepsilon) \cap \Omega_c|$. 
\end{corollary}

To see the corollary, notice that if $\mu(x,\varepsilon) \ge \varepsilon$, then  $\mu(x,\varepsilon) \eta(x,\varepsilon) \lesssim \varepsilon^2$.

\begin{remark}
For usual Beurling-Carleson sets, one has $\varepsilon^2$ in the denominator of (\ref{eq:general-thickness}).
This is essentially the Rodin-Warschawski condition on the existence of a non-zero angular derivative of a Riemann map $\psi_c: \Omega_c \to \mathbb{D}$
  at $x \in \partial \Omega_c \cap \partial \mathbb{D}$, cf.~Theorem \ref{rodin-warschawski}. (If $\Omega_c$ is disconnected, then we consider the Riemann map from an appropriate connected component of $\Omega_c$.) For an application to critical values of inner functions, see \cite{IK22}.
  For $\alpha$-Beurling-Carleson sets,  the denominator of (\ref{eq:general-thickness}) is $\varepsilon^{3-\alpha}$.
\end{remark}

\section{A corona construction}
\label{sec:corona-construction}

In this section, we explore a number of conditions which guarantee that a singular measure is supported on a countable union of Beurling-Carleson sets and prove Theorems \ref{N-theory} and \ref{Hp-theory}. Our main tool is a corona-type decomposition for singular measures.

\subsection{Decomposition of singular measures}
\label{sec:corona-decomposition}

Suppose $\mu$ is a singular measure on the unit circle. 
Fix a large constant $M > 0$ and consider the following corona-type decomposition. Let $\{ I_j ^{(1)}\}$ be the maximal (closed) dyadic arcs such that
$$
\frac{\mu(I_j^{(1)})}{|I_j^{(1)}|} \ge M.
$$
In each $I_j^{(1)}$, we consider the maximal dyadic subarcs $J_k^{(1)} \subset I_j^{(1)}$ for which
$$
\frac{\mu(J_k^{(1)})}{|J_k^{(1)}|} \le \frac{M}{100}.
$$
In each $J_k^{(1)}$, we consider the maximal dyadic subarcs $I_j^{(2)} \subset J_k^{(1)}$ with
$$
\frac{\mu(I_j^{(2)})}{|I_j^{(2)}|} \ge M.
$$
Continuing in this way, we inductively define $I_j^{(m)}$ and $J_k^{(m)}$ for $m \ge 1$. 
We call the arcs $I_j^{(m)}$ {\em heavy}  
%since $$\frac{\mu( I_j^{(m)} )}{ |  I_j^{(m)}|} \ge M$$ 
and the arcs $J_k^{(m)}$  {\em light}\,, $j,k,m \ge 1$.
%as 
%$$
%\frac{\mu( J_k^{(m)} )}{ |  J_k^{(m)}|}  \le \frac{M}{100}.
%$$

Since $\mu$ is a singular measure, almost every point on the unit circle with respect to the Lebesgue measure is eventually contained in a light arc, so that
$$
\sum_{J_k^{(m)} \subset I_j^{(m)}} |J_k^{(m)}| = |I_j^{(m)}|, \qquad j, m \ge 1.
$$
From the definitions of light and heavy arcs, we have
$$
\sum_{I_j^{(m+1)} \subset J_k^{(m)}} |I_j^{(m+1)}| \, \le \, \frac{1}{M} \cdot \mu(J_k^{(m)}) \, \le \, \frac{|J_k^{(m)}|}{100}, \qquad  k, m \ge 1.
$$
It follows that $\mu$ is concentrated on
$$
\bigcup_{ I_j^{(m)} \text{ heavy }} \biggl ( {I}_j^{(m)} \setminus  \bigcup_{\light J_k^{(m)} \subset I_j^{(m)}  } \interior J_k^{(m)} \biggr ).
$$

\subsection{Proofs of Theorems \ref{N-theory} and \ref{Hp-theory}}

For convenience of the reader, we break the proofs of Theorems \ref{N-theory} and \ref{Hp-theory} into two lemmas:

\begin{lemma}
\label{area-condition-countable-unions}
{\em (i)} Let $\mu \ge 0$ be a finite singular measure on $\partial \mathbb{D}$ which satisfies
\begin{equation}
%\label{eq:area-condition}
\int_{\{z \in \mathbb{D}: P_\mu(z) > c\}} \frac{dA(z)}{1-|z|} < \infty,
\end{equation}
for some $c \in \mathbb{R}$.
Then $\mu$ is concentrated on a countable union of Beurling-Carleson sets.

{\em (ii)} Let $\mu \ge 0$ be a finite singular measure on $\partial \mathbb{D}$ which satisfies
\begin{equation}
%\label{eq:area-condition2}
\int_{\{z \in \mathbb{D}: P_\mu(z) > c\}} \frac{dA(z)}{(1-|z|)^{1+p}} < \infty,
\end{equation}
for some $c \in \mathbb{R}$.
Then $\mu$ is concentrated on a countable union of $(1-p)$-Beurling-Carleson sets.
\end{lemma}

\begin{proof}
We only prove (i) as (ii) is similar. We use the decomposition from Section \ref{sec:corona-decomposition}.
 To prove the theorem, it suffices to show that for each heavy interval $I_j^{(m)}$,
$$
E = I_j^{(m)} \setminus  \bigcup_{ \light J_k^{(m)} \subset I_j^{(m)} } \interior J_k^{(m)}
$$
is a Beurling-Carleson set. By Lemma \ref{dyadic-interval-sums}, we may check that
$$
\sum_{\substack{I \dyadic \\   I \cap E \ne \emptyset}} |I| < \infty.
$$
By construction, if $I$ is a dyadic interval in $I_j^{(m)}$  which meets $E$, then $\frac{\mu(I)}{|I|} > \frac{M}{100}$ and $P_\mu(z) \gtrsim M$ for $z \in T_I$. Hence,
$$
\sum_{\substack{I \dyadic \\   I \cap E \ne \emptyset}} |I| \, \lesssim \, \int_{\{z: P_\mu(z) \gtrsim M\}} \frac{dA(z)}{1-|z|} \, < \, \infty
$$
as desired. The proof is complete.
\end{proof}

Ahern and Clark gave an elegant formula for the angular derivative of a singular inner function on the unit circle:
$$
|S'_\mu(z)| =  2 \int_{\partial \mathbb{D}} \frac{d\mu(\zeta)}{|\zeta -  z |^2}, \qquad |z|=1,
$$
where at a given point $z \in \partial \mathbb{D}$, either both quantities are finite and equal or infinite. For a proof, see \cite[Chapter 4.1]{mashreghi}.
\begin{lemma}
\label{derivatives-hp-n-area}
{\em (i)} If $S'_{\mu} \in \mathcal N$, then the area condition (\ref{eq:area-condition}) holds.

{\em (ii)}  If $S'_{\mu} \in H^p$, then the $(1+p)$-area condition (\ref{eq:area-condition2}) holds.
\end{lemma}

\begin{proof}
Observe that 
$$
\Omega_c = \{z \in \mathbb{D}: P_\mu(z) > c\} = \{z \in \mathbb{D} : |S_{\mu}(z)| < e^{-c}\}.
$$
Let $e^{i\theta} \in \partial \mathbb{D}$ be a point at which $S_{\mu}$ has a finite angular derivative. According to a well known result of Ahern and Clark
\cite[Theorem 4.15]{mashreghi},
$$
|S'_\mu(re^{i\theta})| \le 4 |S'_\mu(e^{i\theta})|, \qquad 0 < r < 1.
$$ 
Let $[0,e^{i\theta}]$ denote the radial line segment from the origin to $e^{i\theta}$.
As $1- |S_\mu(re^{i\theta})| \le 4|S'_\mu(e^{i\theta})|(1-r)$,
$$\Omega_c \cap [0, e^{i\theta}] \subset \biggl [0, \biggl (1 - \frac{\varepsilon}{|S'_{\mu}(e^{i\theta})|} \biggr ) \cdot e^{i\theta} \biggl ],
$$
where $\varepsilon > 0$ is a constant that depends on $c$. From this bound on $\Omega_c$, (i) and (ii) follow quite easily.
\end{proof}

\section{Derivative in Hardy spaces}
\label{sec:derivative-hardy}

In this section, we explore conditions on a singular measure $\mu$ involving  Beurling-Carleson sets that guarantee the membership of $S'_\mu$ in $H^p$. We show:

\begin{theorem}
\label{derivative-in-hardy}
Fix $0 < p < 1/2$. Let  $\mu$ be a positive measure supported on a closed set $E \subset \partial \mathbb{D}$ of zero length whose complementary arcs $\{J\}$ satisfy 
\begin{equation}
\label{eq:derivative-in-hardy}
\sum |J|^{1-q} < \infty
\end{equation}
for some $q > \frac{p}{1-p}$. Then, $S'_\mu \in H^p$. 
\end{theorem}

We will give two examples that show that the exponent $\frac{p}{1-p}$ in the theorem above is sharp. Theorem \ref{derivative-in-hardy} improves a result of M.~Cullen \cite{cullen} who showed that $S'_{\mu} \in H^p$ under the stronger hypothesis $q = 2p$.

\subsection[When is S' in Hp?]{When is $S'_\mu \in H^p$?}

We begin by giving a simple criterion for a singular inner function to have derivative in $H^p$. As is standard, for an arc $J$ on the unit circle with $|J| \le 1$, we write
 $z_J = (1-|J|/2) \cdot e^{i \theta_J}$ where  $e^{i \theta_J}$ is the midpoint of $J$.
For $0 < \beta < 1/|J|$, we write $\beta J$ for the arc of length $|\beta J|$ with the same midpoint as $J$.

\begin{lemma}
\label{hp-test}
Fix $0 < p < 1/2$. Suppose $E \subset \partial \mathbb{D}$ is a closed set of zero length and $\{ J \}$ be its complementary arcs. For a positive  measure $\mu$ supported on $E$, we have  $S'_{\mu} \in H^p$ if and only if
\begin{equation}
\label{eq:hp-test}
\sum u(z_J)^p |J|^{1-p} < \infty,
\end{equation}
where $u$ is the Poisson integral of $\mu$.
\end{lemma}

\begin{proof}
Differentiation shows that $S'_{\mu}(z) = h(z) S_{\mu}(z)$, where
$$
h(z) \, = \,
 \int_E \frac{-2\zeta}{(\zeta - z)^2} d\mu(\zeta)
  \, = \, - \int_{E} \frac{2\zeta }{|\zeta - z |^2}  \biggl (  \frac{\overline{\zeta}-\overline{z}}{\zeta-z} \biggr ) d\mu(\zeta).
$$
Notice that if $z/|z| \in J/2$, $|z| \ge 1 - |J|/4$ and $\zeta \in E$, then the quantity
$$
\zeta \cdot  \frac{\overline{\zeta}-\overline{z}}{\zeta-z} = \frac{1-\overline{z}\zeta}{\zeta-z}
$$
is constrained in a sector of aperture strictly less than $\pi$. This tells us that
$$
|h(z)| \, \asymp \, \int_E \frac{d\mu(\zeta)}{|\zeta -  z |^2}
\, \asymp \, \int_E \frac{d\mu(\zeta)}{|\zeta - z_J|^2} 
\, \asymp \, \frac{u(z_J)}{|J|}.
$$
We see that
$$
\int_{J/2} |S'_{\mu}(z)|^p \, |dz| \asymp u(z_J)^p |J|^{1-p},
$$
so the condition (\ref{eq:hp-test}) is necessary for $S'_\mu \in H^p$.

To prove the converse implication, we split $J = \bigcup_{k \in \mathbb{Z}} J_k$ into countably many Whitney arcs such that
$$|J_k| \asymp \dist(J_k, \partial \mathbb{D} \setminus J) \asymp 2^{-|k|}|J|.$$
For $z \in J_k$, we have
$$
| S'_{\mu}(z) | \, = \, 2 \int_E \frac{d\mu(\zeta)}{|\zeta - z|^2}
\, \asymp \, \frac{u(z_{J_k})}{|J_k|}.
$$
 By Harnack's inequality,
$$
 \frac{|J_k|}{|J|} \, \lesssim \, \frac{u(z_{J_k})}{u(z_J)} \, \lesssim \, \frac{|J|}{|J_k|}.
$$
Therefore,
\begin{align*}
\int_J |S'_{\mu}(z)|^p \,|dz| & \lesssim  \sum_k  |J_k| \cdot \frac{u(z_{J_k})^p}{|J_k|^{p}} \\
& \lesssim  u(z_J)^p |J|^{p} \sum_k  |J_k|^{1-2p} \\
& \asymp u(z_J)^p |J|^{1-p}.
\end{align*}
Summing over $J$ shows that $S'_{\mu} \in H^p$.
\end{proof}

With help of Lemma \ref{hp-test}, the proof of Theorem \ref{derivative-in-hardy} runs as follows:

\begin{proof}[Proof of Theorem \ref{derivative-in-hardy}]
Let $u$ be the Poisson integral of $\mu$. Since $u$ is a positive harmonic function, its non-tangential maximal function is in $L^\delta$ for any $\delta < 1$. In particular, for any $\delta < 1$, we have
$$
\sum_J u(z_J)^{\delta} |J| < \infty.
$$
Applying H\"older's inequality with exponents $\delta/p > 1$ and $\delta/(\delta-p) > 1$, we obtain
\begin{align*}
\sum_J u(z_J)^p |J|^{1-p} & = \sum_J u(z_J)^p |J|^{\frac{p}{\delta}}  \, \cdot \, |J|^{\frac{\delta-p}{\delta}-p} \\ 
& \le  \biggl ( \sum_J u(z_J)^{\delta} |J| \biggr)^{\frac{p}{\delta}} 
\biggl ( \sum_J |J|^{1-\frac{\delta p}{\delta - p}}\biggr )^{\frac{\delta-p}{\delta}}.
\end{align*}
Choosing $\delta \in (p, 1)$ so that $\frac{ \delta p }{\delta - p} = q$ gives %, i.e.~$\delta = \frac{pq}{q-p}$ gives
$$
\sum_J u(z_J)^p |J|^{1-p} < \infty,
$$
which implies that $S'_{\mu} \in H^p$ by Lemma \ref{hp-test}. Note that as $\delta$ varies over $(p, 1)$, $q = \frac{ \delta p }{\delta - p} = \frac{1}{1/p - 1/\delta}$ varies over $(\frac{p}{1-p},\infty)$.
\end{proof}

Next, we extend Theorem \ref{derivative-in-hardy} to inner functions:
%Given  an arc $J \subset \partial \mathbb{D}$, let $T(J)$ denote the tent over $J$, that is, the set of points in the disk not contained in any Stolz angle whose vertex is in $\partial \mathbb{D} \setminus J$.

\begin{corollary}
Fix $0 < p < 1/2$. Let $E \subset \partial \mathbb{D}$ be a closed set of zero length whose complementary arcs $\{ J \}$ satisfy
$$
\sum |J|^{1-q} < \infty,
$$
for some $q > \frac{p}{1-p}$. Let $F$ be an inner function whose singular part is supported on $E$ and whose zeros are contained in $K_E$. Then $F' \in H^p$.
\end{corollary} % $K_E = \overline{\mathbb{D}} \setminus \bigcup T(J)$. 

\begin{proof}
By an approximation argument, we can assume that $F$ is a finite Blaschke product with zeros $\{z_n\} \subset K_E$. For each zero $z_n$ of $F$, pick a point $z_n^*$ in $E$ that is closest to $z_n$. Then,
$$
|F'(e^{i\theta})| \, = \, \sum \frac{1-|z_n|^2}{|e^{i\theta} - z_n|^2} \lesssim \,  \sum \, \frac{1-|z_n|^2}{|e^{i\theta} - z_n^*|^2} \, = \, \frac{|S_{\sigma}'(e^{i\theta})|}{2}, 
\quad e^{i\theta} \in  \partial \mathbb{D} \setminus E,
$$
where $\sigma = \sum (1 - |z_n|^2) \delta_{z_n^*}$. From Theorem \ref{derivative-in-hardy}, we know that $S'_\sigma \in H^p$, and by the above equation, $F' \in H^p$ as well.
\end{proof}

\subsection{Sharpness}

We now give two examples showing that the exponent in Theorem \ref{derivative-in-hardy} is sharp:

\begin{lemma} There exists a measure $\mu$ supported on a closed set $E$ of zero length whose complementary arcs $\{J\}$ satisfy
$
\sum |J|^{1 - \frac{p}{1-p}} < \infty
$
yet $S'_{\mu} \notin H^p$. 
\end{lemma}

\begin{proof}
{\em Step 1.} In our example, $E$ will be a certain pruned Cantor set and 
$$\mu = \sum |J|^{\frac{1-2p}{1-p}} (\delta_{a(J)} + \delta_{b(J)}),$$
where $a(J)$ and $b(J)$ are the two endpoints of the complementary arc $J$.
In order for the measure $\mu$ to be finite, we need to arrange that
\begin{equation}
\label{eq:c2}
\sum |J|^{\frac{1-2p}{1-p}} < \infty.
\end{equation}
In addition, we will arrange that
\begin{equation}
\label{eq:c1}
\sum_J \mu(\beta J)^{p} |J|^{1-2p} = \infty,
\end{equation}
for some constant $\beta > 1$ to be chosen.
As $P_\mu(z_J) \gtrsim \frac{\mu(\beta J)}{|J|}$,
$$
\sum_J P_\mu(z_J)^{p} |J|^{1-p} = \infty
$$
and $S'_\mu \notin H^{p}$ by Lemma \ref{hp-test}. 

\medskip

{\em Step 2.} Let $N_j = \# \{ J : |J|  \asymp 2^{-j} \}$. To achieve (\ref{eq:c2}), we request that $N_j \asymp j^{-\alpha} \cdot 2^{\frac{1-2p}{1-p} \cdot j}$ for some $\alpha > 1$ to be chosen. In this case, the total measure supported on the endpoints of arcs of length $\le 2^{-j}$ is
$$
M_j  \, = \, \sum_{|J| \le 2^{-j}} \mu(\overline{J})
 \, \asymp \, \sum_{k=j}^\infty 2^{-\frac{1-2p}{1-p} \cdot k} N_k 
 \, \asymp \, \sum_{k=j}^\infty \frac{1}{k^\alpha}
 \, \asymp \, \frac{1}{j^{\alpha -1}}.
$$
Therefore, if we construct the arcs $\{ J \}$ so that
\begin{equation}
\label{eq:c3}
\mu(\beta J ) \asymp \frac{M_j}{N_j},  \qquad \text{for } |J| \asymp 2^{-j},
\end{equation}
then we would have
$$
\sum_J \mu(\beta J)^{p} |J|^{1-2p} \, \asymp \, \sum_{j=1}^\infty N_j 2^{-j(1-2p)} \biggl ( \frac {M_j}{N_j} \biggr )^{p} \, \asymp \, \sum_{j=1}^\infty \frac{1}{j^{\alpha -p}}.
$$
In order to obtain (\ref{eq:c1}), we may choose $\alpha$ to be any number in $(1,1+p)$.

\medskip

{\em  Step 3.} Fix a real number $A > 2$. Consider the standard Cantor set $E$, which at generation $n$ is formed from $2^n$ arcs of length $A^{-n}$. Inspection shows that $N_j \asymp 2^{j/\log_2 A}$. We choose $A$ appropriately so that
 $$
 \frac{1}{\log_2 A} = \frac{1-2p}{1-p} \in (0,1).
 $$
In order to make $N_j$ smaller, we slightly modify the construction of the standard Cantor set by removing a number of arcs. We call a generation {\em bad} if $N_j > j^{-\alpha} \cdot 2^{\frac{1-2p}{1-p} \cdot j}$ is too large. In a bad generation, we allow each arc to only have one descendant instead of two, say the left one.
In the pruned Cantor set, we have $N_j \asymp j^{-\alpha} \cdot 2^{\frac{1-2p}{1-p} \cdot j}$ as desired. 

We select $\beta > \frac{1}{1-2A}$ so that if $J$ is a complementary arc of some generation, then $\beta J$ covers the interval defining the Cantor set of the previous generation.
Since the mass of $\mu$ is evenly spread out,
$\mu$ satisfies (\ref{eq:c3}).
\end{proof}

In our second example of the sharpness of the exponent in Theorem \ref{derivative-in-hardy}, we have a slightly stronger assumption and a slightly stronger conclusion:
 
\begin{lemma}
Given $q < \frac{p}{1-p}$, there exists a $(1-q)$-Beurling-Carleson set $E$ and a measure $\mu$ supported on $E$ such that $S'_\nu \notin H^p$ for any $0 < \nu \le \mu$. 
\end{lemma}

\begin{proof}
Fix a real number $A > 2$. Consider the standard Cantor set $E$, which at generation $n$ is formed from $2^n$ arcs of length $A^{-n}$. Let $\mu$ be the standard Cantor measure on $E$, that is, $\mu$ is the probability measure supported on $E$ which gives equal mass to arcs of generation $n$.

\medskip

{\em Step 1.~When is $E$ a Beurling-Carleson set?} 
In generation $n$, there are $2^{n-1}$ complementary arcs of length $A^{-n+1}(1-2A^{-1})$. If $\partial \mathbb{D} \setminus E = \bigcup I_k$, then
$$
\sum |I_k|^{1-q} \asymp \sum_n 2^n A^{-(1-q)n},
$$
which converges if $\log A > (\log 2) / (1-q)$. In other words, $E$ is a $q$-Beurling-Carleson set when $\log A > (\log 2) / (1-q)$.

\medskip

{\em Step 2.~When is the measure $\mu$ invisible?} Fix a measure $0 < \nu \le \mu$. Let $\mathcal A(n)$ be the collection of arcs $I$ of generation $n$, in the construction of the Cantor set $E$, such that $\nu(I) \ge 2^{-n-1} \nu(\partial \mathbb{D})$. Since $\# \mathcal A(n) \le 2^n$, we have
$$
\nu(\partial \mathbb{D}) \, \le \, 
\sum_{I \in \mathcal A(n)} \nu(I) + \sum_{I \not\in \mathcal A(n)} \nu(I) \, \le \, 
 \sum_{I \in \mathcal A(n)} \nu(I) + \frac{\nu(\partial \mathbb{D})}{2},
$$
which simplifies to
$$
\sum_{I \in \mathcal A(n)} \nu(I) \ge \frac{\nu(\partial \mathbb{D})}{2}.
$$
However, as $ \nu(I) \le 2^{-n}$ for any $I \in \mathcal A(n)$,
$$\#\mathcal A(n) \ge 2^n \cdot \frac{\nu (\partial \mathbb{D})}{2}.$$
Hence,
\begin{align*}
\sum_{I \in \mathcal A(n)} |I|^{1-p} P_\nu(z_I)^p & \gtrsim \sum_{I \in \mathcal A(n)} |I|^{1-2p} \nu(I)^p \\ 
& \gtrsim 2^n \nu(\partial \mathbb{D}) A^{-n(1-2p)} 2^{-np} \\ 
& = \biggl ( \frac{2^{1-p}}{A^{1-2p}} \biggr)^n \nu(\partial \mathbb{D}).
\end{align*}
Since the lengths and locations of the arcs defining $E$ of generation $n$ are comparable to the complementary arcs of generation $n$, we may use Lemma \ref{hp-test} to conclude that $S'_\nu \notin H^p$ if $2^{1-p} > A^{1-2p}$.

\medskip

{\em Step 3. Conclusion.}~To prove the lemma, we need to find an $A > 2$ satisfying
$$
\frac{1}{1-q} \cdot \log 2 < \log A < \frac{1-p}{1-2p} \cdot \log 2,
$$
which is possible if $1-q > \frac{1-2p}{1-p}$, that is, $q < \frac{p}{1-p}$.
\end{proof}

\medskip

\begin{remark}
There may also be an example in the extreme case when $q = \frac{p}{1-p}$.
\end{remark}

\section{Derivative in Hardy spaces II}
\label{sec:besov-spaces}

Suppose $0 < p < 1/2$ and $\mu \ge 0$ is a singular measure on $\partial \mathbb{D}$. Recall that by Theorem \ref{Hp-theory}, if $S'_{\mu} \in  H^p$ then $S_\mu$ satisfies the $(1+p)$-area condition (\ref{eq:area-condition2}).
We now show that if (\ref{eq:area-condition2}) holds, then $\mu = \sum \mu_i$ can be written as a countable sum of measures with $S'_{\mu_i} \in H^p$. In view of the implication $(2) \Rightarrow (3)$ of Theorem \ref{Hp-theory}, it is enough to prove the following lemma:

\begin{lemma}
Fix $0< p < 1/2$. Suppose $\mu$ is a measure supported on a $(1-p)$-Beurling-Carleson set. If $S_\mu$ satisfies the $(1+p)$-area condition (\ref{eq:area-condition2}),  then $S'_\mu \in H^p$.
\end{lemma}

\begin{proof}
Let $E = \supp \mu$ and write $\partial \mathbb{D} \setminus E = \bigcup J_k$. By Lemma \ref{hp-test}, we need to show that
$$
\sum_k P_\mu(z_{J_k})^p |J_k|^{1-p} < \infty.
$$
Since  $\sum |J_k|^{1-p} < \infty$, we only need to show that
$$
\sum_{k:\, P_\mu(z_{J_k}) \ge 1} P_\mu(z_{J_k})^p |J_k|^{1-p} < \infty.
$$
Let $J \subset \partial \mathbb{D}$ be any arc with $J \cap E = \emptyset$. It is easy to see that
$
\frac{P_\mu(z_{I})}{|I|} \gtrsim \frac{P_\mu(z_{J})}{|J|},
$
for any arc $I \subset J$. Therefore, if $P_\mu(z_{J_k}) \ge 1$, then
\begin{align*}
\sum_{\substack{  I \subset J_k \dyadic \\ P_\mu(z_{I}) \ge 1}} |I|^{1-p}  
& \gtrsim
\sum_{\substack{  I \subset J_k \dyadic \\ | I| \gtrsim |J_k| / P_\mu(z_{J_k})}} |I|^{1-p}  \\
& \asymp \sum_{n=0}^{\log_2 P_\mu(z_{J_k})}  2^n  \cdot (2^{-n} |J_k|)^{1-p} \\ 
& \asymp |J_k|^{1-p} P_\mu(z_{J_k})^p.
\end{align*}
By Harnack's inequality, one can find a constant $0 < c < 1$ so that
\begin{align*}
\sum_{k:\, P_\mu(z_{J_k}) \ge 1} P_\mu(z_{J_k})^p |J_k|^{1-p} & \lesssim \sum_{k:\, P_\mu(z_{J_k}) \ge 1} \ 
\sum_{\substack{ I \subset J_k \dyadic \\ P_\mu(z_{I}) \ge 1}} |I|^{1-p} \\
& \lesssim
\int_{\{z \in \mathbb{D}:\, |S_\mu(z)| \le c\}} \frac{dA(z)}{(1-|z|)^{1+p}},
\end{align*}
which is finite by assumption. The proof is complete.
\end{proof}
 
We now give an example of a singular inner function $S_\mu$ which satisfies the $(1+p)$-area condition (\ref{eq:area-condition2}) yet $S'_\mu \notin H^p$.

\begin{lemma}
For $0 < p < 1/2$, there exists a singular inner function $S_\mu$ with $S'_\mu \notin H^p$ such that
\begin{equation*}
\int_{\{z \in \mathbb{D}:\, |S_\mu(z)| \le c\}}  \frac{dA(z)}{(1-|z|)^{1+p}} < \infty,
\end{equation*}
for any $0 < c < 1$.
\end{lemma}

\begin{proof}[Sketch of proof]
 To get a feeling of why the lemma is true, we examine the situation for the measure $\mu$ which consists of $n$ equally spaced point masses on the circle: $\mu = (1/n^{2-\varepsilon}) \sum_{k=0}^{n-1} \delta_{\xi_k}$ where $\xi_k = e^{2\pi i k/n}$, $k  = 0,1,2,\dots, n-1$ and $\varepsilon > 0$ is a constant to be chosen. Since
\begin{align*}
|S'_{\mu}(e^{i\theta})| & =  \int_0^{2\pi} \frac{2d\mu(t)}{|e^{i\theta} - e^{it}|^2}  \\
& =  \frac{2}{n^{2-\varepsilon}} \sum_{k=0}^{n-1} \frac{1}{|e^{i\theta}-\xi_k|^2} \\
& \asymp \, \frac{1}{n^{2-\varepsilon} \cdot \dist(e^{i\theta}, \{\xi_k\})^2},
\end{align*}
the integral
$$
\int_0^{2\pi} |S'_{\mu}(e^{i\theta})|^p d\theta \, \asymp \, n \int_{-\pi/n}^{\pi/n} \biggl ( \frac{1}{n^{2-\varepsilon}\theta^2} \biggr)^p d\theta \, \asymp \, n^{\varepsilon p}
$$
 tends to infinity as $n \to \infty$.

Let $H_k$ be the horoball which rests at $\xi_k$ of diameter $\alpha/n^{2-\varepsilon}$. It is not difficult to see that for any $0 < c < 1$, there exists an
$\alpha = \alpha(c) > 0$ such that
$$
\{ z \in \mathbb{D}: |S_{\mu}(z)| \le c \} \subseteq \bigcup_{k=0}^{n-1} H_k.
$$
As the integral over a single horoball
$$
\int_{H_0} \frac{dA(z)}{(1-|z|^2)^{1+p}} \asymp \frac{1}{n^{(2-\varepsilon)(1-p)}},
$$
the integral over their union is
$$
\int_{\bigcup {H_k}} \frac{dA(z)}{(1-|z|^2)^{1+p}} \asymp n^{1-(2-\varepsilon)(1-p)}.
$$
Since $0 < p < 1/2$, we can choose $\varepsilon > 0$ sufficiently small to make the exponent $1-(2-\varepsilon)(1-p)$ negative, so that the integrals
$$
\int_{\{z\in\mathbb{D}:\, |S_\mu(z)|< c\}} \frac{dA(z)}{(1-|z|)^{1+p}}
$$
tend to 0 as $n \to \infty$.

Independent copies of this construction provide an example of a singular inner function $S$ with $S' \notin H^p$ for which
$$
\int_{\{z\in\mathbb{D}:\, |S(z)|< c\}} \frac{dA(z)}{(1-|z|)^{1+p}} < \infty.
$$
We leave the details to the reader.
\end{proof}

\section{Background on angular derivatives}

For $0 < \theta < \pi$ and $0 < \delta < 1$, let $S_{\theta, \delta}(p) = S_\theta(p) \cap B(p, \delta)$ denote the truncated Stolz angle of opening $\theta$ with vertex at $p \in \partial \mathbb{D}$.

Suppose $\Omega \subset \mathbb{D}$ is a domain in the unit disk bounded by a Jordan curve. We say that $\Omega$ has an {\em inner tangent} at a point $p \in \partial \Omega \cap \partial \mathbb{D}$ if for any $0 < \theta < \pi$, $\Omega$ contains a truncated Stolz angle of opening $\theta$ with vertex at $p$.

Let $\varphi: \mathbb{D} \to \Omega$ be a conformal map.  We say that $\varphi$ has a (non-zero) {\em angular derivative} at $q = \varphi^{-1}(p)$ if the non-tangential limit
$$
\lim_{z \to q} |\varphi'(z)| = A,
$$
for some real number $A > 0$. 
While the number $A$ depends on the choice of Riemann map $\varphi$, the existence of the angular derivative does not. 
In other words, possessing an angular derivative is an intrinsic property of $(\Omega, p)$, which we record by saying that $\Omega$ is {\em thick} at $p$.
In the language of potential theory, one would say that the complement $\mathbb{D} \setminus \Omega$ is minimally thin at $p$, see \cite[Theorem 5.2]{burdzy}, which means that Brownian motion conditioned to exit the unit disk at $p$ is eventually contained in $\Omega$.

To avoid dealing with the point $q$, we will simply say that the inverse conformal map $\psi: \Omega \to \mathbb{D}$ has an angular derivative at $p$ and write $|\psi'(p)| = A^{-1}$. 
It is easy to see that if $\Omega$ is thick at $p$, then $\Omega$ possesses an inner tangent at $p$.

Rodin and Warschawski gave an if and only if condition for $\psi$ to possess an angular derivative at $p$ in terms of moduli of curve families, e.g.~see \cite[Theorem V.5.7]{garnett-marshall} or \cite{BK}. When $\Omega$ is a starlike domain with regular boundary, their condition takes a simpler form \cite{IK22}:

\begin{theorem}
\label{rodin-warschawski}
Suppose $\Omega = \bigl \{ r\zeta : \zeta \in \partial \mathbb{D}, \, 0 \le r < 1 - h(\zeta) \bigr \}$,
where $h: \partial \mathbb{D} \to [0, 1/2]$ is a continuous function. Assume that $h$ satisfies the doubling condition
 $$
 h(\zeta_1) \ge c \cdot h(\zeta_2), \qquad \text{whenever }|\zeta_2 - \zeta_1| < c \cdot h(\zeta_1),
 $$
 for some $c > 0$. Then, $\psi$ has an angular derivative at $p \in \partial \Omega \cap \partial \mathbb{D}$ if and only if
$$
\int_{\partial \mathbb{D}} \frac{h(\zeta)}{|\zeta-p|^2} |d\zeta| < \infty.
$$
 \end{theorem}
 
We will use the following elementary lemma about angular derivatives: %They can be proved using the methods of \cite{IK22}.

\begin{lemma}
\label{increasing-angular-derivatives}
Let $\{\Omega_n\}_{n=1}^\infty$ be an increasing sequence of Jordan domains, whose union is the unit disk. Suppose the conformal maps $\psi_n: \Omega_n \to \mathbb{D}$ converge uniformly on compact subsets to the identity.
If $\psi_1$ has an angular derivative at $p \in \partial \Omega \cap \partial \mathbb{D}$, then the angular derivatives  $|\psi'_n(p)|$ tend to 1.
\end{lemma}

We will also need the following theorem from \cite{IK22} which describes how composition operators %with univalent symbols 
act on measures on the unit circle:
 
\begin{theorem}
\label{radon-nikodym2}
Suppose $\Omega \subset \mathbb{D}$ is a Jordan domain, $\varphi: \mathbb{D} \to \Omega$ is a conformal map
and $\psi: \Omega \to \mathbb{D}$ is its inverse. 
Let $\mu \ge 0$ be a positive measure on the unit circle.  Since $P_{\mu}(\varphi(z))$ is a positive harmonic function, it can be represented as the Poisson extension of some finite measure $\nu \ge 0$. If we use the normalization $0 \in \Omega$ and $\varphi(0) = 0$, then
\begin{equation}
\label{eq:radon-nikodym2}
\varphi_*\nu = P_\mu d\omega_{\Omega, 0} + 
|\psi'| \, d\mu,
\end{equation}
provided that we interpret $|\psi'(p)| = 0$ if $p \notin \partial \Omega$ or $\Omega$ is not thick at $p$.
\end{theorem}

\section{Background in PDE}

\label{sec:background-PDE}

In this section, we make some general observations about semi-linear elliptic equations of the form
\begin{equation}
\label{eq:g-PDE}
\Delta u = g(u),
\end{equation}
which will be used in Section \ref{sec:nearly-maximal}.
We assume that the ``non-linearity'' $g$ is a non-negative increasing convex function which satisfies the Keller-Osserman condition \cite{keller,osserman}
\begin{equation}
\label{eq:keller}
\int_1^\infty \frac{ds}{\sqrt{G(s)}} < \infty,
\end{equation}
where $G' = g$.
Examples of $g$ satisfying the above conditions include $g(t) = e^{2t}$ and $g(t) = t^p \cdot \chi_{t > 0}$ with $p > 1$.

\subsection{Basic properties}

\label{sec:basic-PDE}

\subsubsection*{Traces}

Given a function $\phi$ on the unit disk, we define its {\em boundary trace} as the weak limit of the measures
$\phi(re^{i\theta}) d\theta$ as $r \to 1$, provided that the limit exists. Otherwise, we say that $\phi$ does not possess a boundary trace.

\subsubsection*{Sub- and super-solutions}

  One says that a function $v: \mathbb{D} \to \mathbb{R}$ is a {\em subsolution} of (\ref{eq:g-PDE}) if $\Delta v \ge g(v)$ in the sense of distributions. Similarly, we say that $v$ is a {\em supersolution} if $\Delta v \le g(v)$ in the sense of distributions.
  
  \begin{theorem}[Principle of sub- and supersolutions]
Suppose $u_-$ is a subsolution and $u_+$ is a supersolution of (\ref{eq:g-PDE}) with $u_-(z) \le u_+(z)$ for any $z \in \mathbb{D}$. Then, there exists at least one solution $u(z)$ with
$$
u_-(z) \, \le \, u(z) \, \le \, u_+(z), \qquad z \in \mathbb{D}.
$$
  \end{theorem}

A proof using the Schauder fixed point theorem can be found in \cite[Chapter 20]{ponce-book}.

\subsubsection*{Existence of solutions and the comparison principle}

\begin{theorem}
\label{l1theorem}
Given a function $h \in L^\infty(\partial \mathbb{D})$,  the boundary value problem
\begin{equation}
\label{eq:generalized}
   \left\{\begin{array}{lr}
        \Delta u  =  g(u), &  \text{in } \mathbb{D}, \\
       u = h, &  \text{on } \partial \mathbb{D},
        \end{array}\right.
\end{equation}
admits a unique solution, where the boundary values are interpreted in the sense of weak limits of measures. If $u_1$ and $u_2$ are two solutions with boundary values $h_1 \le h_2$, then  $u_1 \le u_2$ on $\mathbb{D}$.
\end{theorem}

\begin{proof}[Proof of Theorem \ref{l1theorem}] {\em Step 1. Uniqueness and monotonicity.}
By Kato's inequality \cite[Proposition 6.9]{ponce-book},
$$
\Delta (u_1-u_2)^+ \ge \Delta (u_1- u_2) \cdot \chi_{\{u_1 > u_2\}} = (g(u_1) - g(u_2)) \cdot \chi_{\{u_1 > u_2\}} \ge 0
$$
is a subharmonic function. As $h_1 \le h_2$, the function $(u_1-u_2)^+$ has zero boundary values. The maximum principle shows that $(u_1-u_2)^+ \le 0$ or $u_1
 \le u_2$. The same argument also proves uniqueness.

\medskip

{\em Step 2. Existence.}
Let $P_h$ denote the harmonic extension of $h$ to the unit disk. Clearly, $u_+ = P_h$ is a supersolution of (\ref{eq:g-PDE}) with boundary data $h$. Similarly, if  $G(z, w) = \log \bigl | \frac{1-\overline{w}z}{w-z} \bigr |$ is the Green's function of the unit disk, then
\begin{equation*}
u_-(z) = P_h(z) - \frac{1}{2\pi} \int_{\mathbb{D}} g(\| h \|_\infty) G(z, \zeta) dA(\zeta)
\end{equation*}
is a subsolution of (\ref{eq:g-PDE}) as
$$
\Delta u_-(z) =  g(\| h \|_\infty) \ge g(u_-(z)).
$$
Since $u_-$ also has boundary trace $h$, by the principle of sub- and supersolutions, there exists a solution with boundary trace $h$.
\end{proof}

\subsubsection*{The maximal solution}

\begin{lemma}
The PDE (\ref{eq:g-PDE}) has a unique maximal solution $u_{\max}$ on the unit disk, which dominates all other solutions pointwise.
\end{lemma}

\begin{proof}[Sketch of proof]
We will simultaneously show that (\ref{eq:g-PDE}) has a maximal solution on every disk $\mathbb{D}_R = \{z: |z| < R \}$ with $R > 0$.

Keller \cite{keller} and Osserman \cite{osserman} showed that under the assumption (\ref{eq:keller}), for any $R > 0$, there is a unique radially-invariant solution $u_R(z)$ on  $\mathbb{D}_R$ which tends to infinity as $|z| \to R$, and furthermore, the solutions $u_R(z)$ depend continuously on $R$.

Suppose $u: \mathbb{D}_R \to \mathbb{R}$ is any solution of (\ref{eq:g-PDE}). By the comparison principle, for any $S < R$, $u(z) < u_S(z)$ on $\mathbb{D}_S$. Taking $S \to R$ yields $u(z) \le u_R(z)$.
\end{proof}

The above argument shows that if $u$ is a solution of (\ref{eq:g-PDE}) on the unit disk which tends to infinity as $|z| \to 1$, then $u = u_{\max}$. As a consequence, the solutions $u_n$ of  (\ref{eq:g-PDE}) with constant boundary values $n$ increase to $u_{\max}$ as $n \to \infty$.

\begin{remark}
For the existence and uniqueness of large solutions of semi-linear equations on other domains, we refer the reader to \cite{BM,garcia-melian}. Information about the asymptotic behaviour of these solutions near the boundary can be found in \cite{BM2, BM3, DL, LM}. 
\end{remark}

%% --- DETAILS
\begin{comment}

If we search for radially-invariant solutions $u(r) = u(|z|)$, then (\ref{eq:g-PDE}) becomes a second-order ODE
\begin{equation}
\label{eq:g-ODE}
\frac{1}{r} \frac{\partial}{\partial r}(ru_r) = u_{rr} + \frac{u_r}{r} = g(u(r)).
\end{equation}
Since (\ref{eq:g-ODE}) is a second-order ODE, it has a two-parameter family of solutions. However, in order for $u$ to be $C^2$ near the origin, we must insist that $u_r(0) = 0$, which leaves us with a one-parameter family of solutions, parametrized by the value of $u(0) \in \mathbb{R}$.

Define $t_* = \max \{t \in \mathbb{R} : g(t) = 0\}$. Clearly, if $u(0) \le t_*$ then $u(r) = u(0)$ is constant.
Keller \cite{keller} and Osserman \cite{osserman} showed that under the assumption (\ref{eq:keller}), if $u(0) > t_*$
then the ODE  (\ref{eq:g-ODE}) blows up in finite time.
From here, it is not difficult to show that for any $R > 0$, there is a unique radially-invariant solution $u_R(z)$ on  $\mathbb{D}_R$ which tends to infinity as $|z| \to R$, and furthermore, the solutions $u_R(z)$ depend continuously on $R$.
\end{comment}
%%

\subsubsection*{Minimal dominating solution}
 
 Let $v$ be a subsolution of (\ref{eq:g-PDE}).
For $0 < r < 1$, we write
$\Lambda_r[v]$ for the unique solution of  (\ref{eq:g-PDE}) on the disk $\mathbb{D}_r = \{z: |z| < r\}$ which agrees with $v$ on $\partial \mathbb{D}_r$. An inspection of Step 1 of the proof of Theorem \ref{l1theorem} shows that  $\Lambda_r[v]$ is the pointwise-minimal solution which lies above $v$ on $\mathbb{D}_r$. In particular, the solutions $\Lambda_r[v]$ are increasing in $r$. 
The limit $\Lambda[v] := \lim_{r \to 1} \Lambda_r[v]$ is finite on the unit disk because it is bounded above by the maximal solution.

For any test function $\phi \in C_c^\infty(\mathbb{D})$, we have
$$
\int_{\mathbb{D}_r} u_r \Delta \phi \, dA = \int_{\mathbb{D}_r} g(u_r) \phi \, dA, \qquad u_r = \Lambda_r[v],
$$
provided that $\mathbb{D}_r$ contains $\supp \phi$ in its interior. After taking $r \to 1$ and using the dominated convergence theorem, it follows that $\Lambda[v]$ is a solution of (\ref{eq:g-PDE}).
From the construction, it is clear that $\Lambda[v]$ is the pointwise-minimal solution which satisfies $\Lambda[v] \ge v$.

\begin{remark}
This construction generalizes the notion of the {\em minimal harmonic majorant} for subharmonic functions on the unit disk. One small but important difference is that the minimal harmonic majorant does not always exist (i.e.~may be identically $+\infty$).
\end{remark}

\subsection{Nearly-maximal solutions}  

 A solution of (\ref{eq:g-PDE}) is called {\em nearly-maximal} if
\begin{equation}
\limsup_{r \to 1} \int_{|z| = r} (u_{\max} - u) d\theta < \infty.
\end{equation}
For each $0 < r < 1$, we may view $(u_{\max} - u) d\theta$ as a positive measure on the circle of radius $r$. Subharmonicity guarantees the existence of a weak limit as $r \to 1$, so we obtain a measure
$\mu[u]$ on the unit circle associated to $u$. We refer to $\mu$ as the {\em deficiency measure} of $u$.

Notice that if $\mu \ge 0$ is a measure on the unit circle and $P_\mu$ is its Poisson extension to the unit disk, then $\Lambda[u_{\max}-P_\mu]$ is a nearly-maximal solution. Clearly, the deficiency measure $\nu$ of  $\Lambda[u_{\max}-P_\mu]$ is at most $\mu$. %The lemma below shows that a nearly-maximal solution $u$ with deficiency measure $\mu$ exists if and only if  $u = \Lambda[u_{\max} - P_\mu]$.

\begin{lemma}[Fundamental lemma]
\label{fundamental-lemma}
If $u$
 is a nearly-maximal solution of (\ref{eq:g-PDE}) with deficiency measure $\mu$, then $u = \Lambda[u_{\max} - P_\mu]$.
\end{lemma}

\begin{proof}
{\em Step 1.} Observe that $u_{\max} - P_\mu$ is a subsolution since
$$
\Delta (u_{\max} - P_\mu) \, = \, g(u_{\max}) \, \ge \, g(u_{\max} - P_\mu).
$$
We claim that $u \ge u_{\max} - P_\mu$ and thus $u \ge \Lambda [u_{\max} - P_\mu]$.
To this end, we consider the function
$$
\phi := u_{\max} - u - P_\mu.
$$
Since $\phi$ is a subharmonic function with zero boundary trace, by the maximum principle, 
 $\phi \le 0$ in the unit disk.

\medskip

{\em Step 2.} 
As $v := \Lambda[u_{\max} - P_\mu]$ is a nearly-maximal solution, it possesses a deficiency measure $\nu$. 
From Step 1, we know that
$$
u \, \ge \, v \, = \, \Lambda[u_{\max} - P_\mu] \, \ge \, u_{\max} - P_\mu.
$$
After rearranging, we get
$$u_{\max} - u \, \le \, u_{\max} - v \, \le \, P_\mu.$$
Taking the weak limit as $r \to 1$, we see that $\nu = \mu$.

\medskip

{\em Step 3.} Finally, since $u - v$ is a non-negative subharmonic function with zero boundary trace,
$u = v$. 
\end{proof}

In particular, Lemma \ref{fundamental-lemma} shows that the deficiency measure $\mu$ uniquely determines the nearly-maximal solution $u$. Below, we will write $u_{\mu}$ for the nearly-maximal solution associated to the measure $\mu$, if it exists. Another simple consequence of 
Lemma \ref{fundamental-lemma} is the {\em monotonicity principle} for nearly-maximal solutions:

\begin{corollary}[Monotonicity principle]
\label{monotonicity-principle2} If  $\nu < \mu$ then $u_{\nu} > u_{\mu}$.
\end{corollary}

\subsection{Constructible and invisible measures}

We say that a measure $\mu$ on the unit circle is {\em invisible} if for any measure $0 < \nu \le \mu$, there does not exist a nearly-maximal solution $u_{\nu}$ with deficiency measure $\nu$. In this section, we show that any positive measure on the unit circle can be uniquely decomposed into a deficiency measure and an invisible measure. 

\begin{theorem}
\label{nu-mu2}
Suppose $\mu$ is a positive measure on the unit circle. If $u_\nu = \Lambda[u_{\max} - P_\mu]$, then $\nu$ is a deficiency measure and $\mu - \nu$ is an invisible measure.
\end{theorem}

In particular, a measure $\mu$ is invisible if and only if $\Lambda[u_{\max} - P_\mu] = u_{\max}$.
We will break the proof Theorem \ref{nu-mu2} into a series of lemmas.

\begin{lemma}
\label{division-law}
 If a measure $\mu$ is a deficiency measure, then any measure $0 \le \mu_1 \le \mu$ is also a deficiency measure.
 % If $0 \le \nu \le \mu$, and $\mu$ is a deficiency measure of some nearly-maximal solution, then so is $\nu$.
\end{lemma}

\begin{proof} To show that $\mu_1$ is a deficiency measure, we check that  $\mu_1 = \nu_1$ where $u_{\nu_1} = \Lambda[u_{\max}-P_{\mu_1}]$. Since the inequality $\nu_1 \le \mu_1$ is always true, we only need to prove the opposite inequality $\mu_1 \le \nu_1$.

Let $\mu_2 = \mu - \mu_1$.
 Using the same argument as in the proof of Lemma \ref{fundamental-lemma}, it is not difficult to show that $$
\Lambda[u_{\max}-P_{\mu_1+\mu_2}] \ge \Lambda[u_{\max}-P_{\mu_1}] - P_{\mu_2}
$$
or
$$
u_{\max} - \Lambda[u_{\max}-P_{\mu_1+\mu_2}] \le u_{\max} -
\Lambda[u_{\max}-P_{\mu_1}] + P_{\mu_2}.
$$
Taking traces, we see that 
$\mu_1 + \mu_2 \le \nu_1 + \mu_2$ or 
 $\mu_1 \le \nu_1$  as desired.
\end{proof}

\begin{lemma}
\label{product-law}
{\em (i)} The sum of two deficiency measures is a deficiency measure. 

{\em (ii)} Suppose $\mu_i$, $i=1,2,3,\dots$ are deficiency measures such that their sum $\mu = \sum \mu_i$ is a finite measure. Then, $\mu$ is also a deficiency measure.
\end{lemma}

%%%
In the proof below, we will use the following elementary observation: if $g$ is a convex function and $x_1 < x_2 < x_3 < x_4$ are four real numbers satisfying $x_1 + x_4 = x_2 + x_3$, then
\begin{equation}
\label{eq:g-inequality}
g(x_2) + g(x_3) \le g(x_1) + g(x_4).
\end{equation}
Moreover, if $g$ is an {\em increasing} convex function, then (\ref{eq:g-inequality}) holds under the weaker assumption $x_1 + x_4 \ge x_2 + x_3$. This is a one-dimensional analogue of the fact that the composition $\phi \circ u$ of an increasing convex function $\phi$ and a subharmonic function $u$ is subharmonic.
%%%

\begin{proof}[Proof of Lemma \ref{product-law}]
(i) Suppose $\mu = \mu_1 + \mu_2$ is a measure on the unit circle. Set 
$$
u_\nu = \Lambda[u_{\max} - P_\mu].
$$
In view of the discussion preceding Lemma \ref{fundamental-lemma}, to prove (i), it is enough to show that
\begin{equation}
\label{eq:product-lemma}
 \mu_1 + \mu_2 \le \nu.
\end{equation}
To verify (\ref{eq:product-lemma}), we check that
$$
\Lambda[u_{\max} - P_{\mu_1}] + \Lambda[u_{\max} - P_{\mu_2}] \ge \Lambda[u_{\max}] + \Lambda[u_{\max} - P_{\mu}],
$$
which we abbreviate to  $B + C \ge A + D$. Clearly, $A \ge B \ge D$ and $A \ge C \ge D$.
Consider the function
$$\phi = (A + D - B - C)^+.$$
Since $g$ is an increasing convex function, at a point $z \in \mathbb{D}$ where $A + D > B + C$, we have
$$\Delta \phi(z) = g(A(z)) + g(D(z)) - g(B(z)) - g(C(z)) \ge 0.$$
In view of Kato's inequality, $\phi$ is subharmonic and non-negative on the unit disk. If we knew that $\phi$ had zero trace, then we could immediately say that $\phi$ is identically 0.

Due to difficulties examining the trace of $\phi$ on $\partial \mathbb{D}$ directly, we use an approximation argument. For each $0 < r < 1$, we consider the function 
$$
\phi_r \, = \, \Bigl ( \Lambda_r[u_{\max} - P_{\mu_1}] + \Lambda_r[u_{\max} - P_{\mu_2}] - \Lambda_r[u_{\max}] - \Lambda_r[u_{\max} - P_{\mu}] \Bigr)^+
,
$$
defined on $\mathbb{D}_r$. The above argument shows that $\phi_r$ is a non-negative subharmonic function on $\mathbb{D}_r$.
As $\phi_r$ has zero boundary values on $\partial \mathbb{D}_r$,  it is identically 0. Taking $r \to 1$, we see that $\phi$ is identically 0 as desired.

(ii) Set  $\tilde \mu_j =  \mu_1 + \mu_2 + \dots + \mu_j$. By part (i), we have
$$
\Lambda[u_{\max} - P_{\mu}] \le \Lambda[u_{\max} - P_{\tilde \mu_j}]  = u_{\tilde \mu_j}.
$$
The above equation shows that if
$$
 u_\nu = \Lambda[u_{\max} - P_{\mu}],
$$
then $\nu \ge \tilde \mu_j$ for any $j$, which implies that $\nu \ge \mu$. As the reverse inequality is always true, $\nu = \mu$ as desired.
\end{proof}

\begin{lemma}
If $\mu \ge 0$ is a measure on the unit circle and $u_\nu = \Lambda[u_{\max} - P_{\mu}]$, then the difference $\mu - \nu$ is invisible.
\end{lemma}

\begin{proof}
We need to show that any measure $0 < \omega \le \mu - \nu$ does not arise as a deficiency measure of some nearly-maximal solution.
The existence of $u_\omega$ would imply the existence of $u_{\nu + \omega}$ by Lemma \ref{product-law}, which would in turn lead to the estimate
$$
u_{\max} - P_\mu \, \le u_{\max} - P_{\nu + \omega} \, \le  \, u_{\nu+\omega} \, \le \, u_\nu,
$$
by the monotonicity principle and the fundamental lemma (Lemmas \ref{monotonicity-principle2} and \ref{fundamental-lemma} respectively). This contradicts the definition of $u_\nu$ as the {\em least} solution that lies above $u_{\max} - P_{\mu}$.
\end{proof}

\subsection{A lemma on iterated majorants}

For future reference, we record the following lemma:

\begin{lemma}
\label{one-at-a-time}
 {\em (i)} For two positive measures $\mu_1$ and $\mu_2$ on the unit circle,
$$
\Lambda \Bigl [  \Lambda \bigl [u_{\max} - P_{\mu_2} \bigr ] - P_{\mu_1} \Bigr] = \Lambda \bigl [ u_{\max} - P_{\mu_1+\mu_2} \bigr ].
$$

{\em (ii)} More generally,
$$
\Lambda \Bigl [ \dots \Lambda \bigl [ \Lambda[u_{\max} - P_{\mu_j}] - P_{\mu_{j-1}} \bigr ]  \dots - P_{\mu_1} \Bigr  ] = \Lambda \bigl [ u_{\max} - P_{\mu_1+\mu_2+\dots+\mu_j} \bigr ].
$$

{\em (iii)} If $\mu = \sum_{j=1}^\infty \mu_j$ is a finite measure, then
$$
\lim_{j \to \infty}\Lambda \Bigl [ \dots \Lambda \bigl [ \Lambda[u_{\max} - P_{\mu_j}] - P_{\mu_{j-1}} \bigr ] \dots - P_{\mu_1} \Bigr  ]
= \Lambda \bigl [ u_{\max} - P_{\mu} \bigr ],
$$
pointwise on the unit disk.
\end{lemma}

\begin{proof}
(i) The $\ge$ direction follows from the monotonicity of $\Lambda$. For the $\le$ direction, it suffices to show that
$$
 \Lambda \bigl [u_{\max} -P_{\mu_2} \bigr ] - P_{\mu_1} \le \Lambda \bigl [  u_{\max} - P_{\mu_1+\mu_2}  \bigr ]
$$
or
$$
 \Lambda_r \bigl [ u_{\max} - P_{\mu_2}\bigr ] - P_{\mu_1}  \le \Lambda_r \bigl [  u_{\max} - P_{\mu_1+\mu_2} \bigr ]
$$
for any $0 < r < 1$. To this end, we form the function
$$
u_r =  \Bigl ( \Lambda_r \bigl [ u_{\max} - P_{\mu_2} \bigr ] - P_{\mu_1} \Bigr ) - \Lambda_r \bigl [  u_{\max} - P_{\mu_1+\mu_2} \bigr ],
$$
defined on $\mathbb{D}_r = \{ z : |z| < r \}$. Since $u_r$ is subharmonic and vanishes on $\partial \mathbb{D}_r$, it must be identically 0. This proves the $\le$ direction.

(ii) follows after applying (i) $j-1$ times.

(iii) Let $\tilde \mu_j = \mu_1 + \mu_2 + \dots + \mu_j$. By part (i), we have
 $$
 \Lambda \bigl [u_{\max} - P_{\tilde \mu_j} \bigr ] - P_{\mu-\tilde \mu_j} \, \le \,
  \Lambda \bigl [ u_{\max} - P_\mu \bigr ] \, \le \,  \Lambda \bigl [ u_{\max} - P_{\tilde \mu_j} \bigr ].
 $$
 Since $P_{\mu-\tilde \mu_j} \to 0$ pointwise in the unit disk, the minimal dominating solutions  $\Lambda \bigl [u_{\max} - P_{\tilde \mu_j} \bigr ]$
decrease to $\Lambda \bigl [u_{\max} - P_{\mu}\bigr ]$.
\end{proof}

\section{Nearly maximal solutions}
\label{sec:nearly-maximal}

In this section, we prove Theorem \ref{nearly-maximal} which partially characterizes the nearly-maximal solutions of
\begin{equation}
\label{eq:p-PDE}
\Delta u = u^p \cdot \chi_{u > 0}, \qquad \text{on }\mathbb{D},
\end{equation}
with $p > 1$. From Section \ref{sec:basic-PDE}, we know that (\ref{eq:p-PDE}) has a radially invariant solution $u_{\max}$ which dominates all the other solutions pointwise. By solving an ODE, one can write down an explicit formula for $u_{\max}$. Here, we will only need the asymptotic formula
$$
u_{\max}(z) \sim C_\alpha (1-|z|)^{\alpha-1}, \qquad |z| \to 1,
$$
where $\alpha = \frac{p-3}{p-1}$. We will be especially interested in the case when $p > 3$, in which case $\alpha \in (0,1)$.

The proof of Theorem \ref{nearly-maximal} consists of two parts:
\begin{enumerate}
\item First, we show that if $\mu$ does not charge $\alpha$-Beurling-Carleson sets, then it is not the deficiency measure of any nearly-maximal solution. 
As the proof is similar to the one in  \cite{ivr19} for $\Delta u = e^{2u}$, we only give a sketch of the argument in Section \ref{sec:invisible}.

\item Secondly, we show that if $\mu$ is concentrated on an $\beta$-Beurling-Carleson set, for some $\beta < \alpha$, then there is a nearly-maximal solution $u_\mu$ with deficiency measure $\mu$. The argument in \cite{ivr19} relied on the Liouville correspondence between solutions of $\Delta u = e^{2u}$ and holomorphic self-mappings of the disk, which is unavailable in the present setting. We present a new approach to existence which involves  special  Privalov stars with round corners. The special Privalov stars will be constructed in Section \ref{sec:special-star} and the existence will be explained in Section \ref{sec:existence}.
\end{enumerate}

\subsection{Restoring property}

We focus on the case when $p > 3$.
The following lemmas will be used in conjunction with Roberts decompositions to show that certain measures on the unit circle are invisible:
 
\begin{lemma}
\label{fail-lemma}
Let $n_i = 2^i$. For any $0 < a < 1$, there exists $a < b < 1$ such that
\begin{equation}
\label{eq:fail-lemma}
\Lambda_{1-1/n_{i+1}}[a \cdot u_{\max}] > b \cdot u_{\max}, \qquad \text{on }\{z : |z| = 1 - 1/n_i\}.
\end{equation}
 \end{lemma}

\begin{proof}
We prefer to work on the upper half-plane $\mathbb{H}$ since the expression for the maximal solution is simpler there: $u_{\max}(z) = C_\alpha y^{\alpha-1}$ where $y = \im z$. We need to show that
$$
\Lambda_{y_0}[a \cdot u_{\max}] > b \cdot u_{\max}, \qquad \text{on }\{\im z = 2y_0\}.
$$

 When extending constant boundary values from a horizontal line, we get the maximal solution shifted vertically by an appropriate amount:
$$u = \Lambda_{y_0}[a \cdot u_{\max}]  = C_\alpha (y+c)^{\alpha-1},$$
where $c$ is determined by the equation
$$
a \cdot C_\alpha y_0^{\alpha-1} = C_\alpha (y_0+c)^{\alpha-1} \quad  \implies \quad c =  a^{\frac{1}{\alpha-1}} \cdot y_0 - y_0.
$$
In particular,
$$
u(2y_0) = C_\alpha (1 + a^{\frac{1}{\alpha-1}})^{\alpha -1} \cdot y_0^{\alpha-1}.
$$
This suggests that we should take 
$$
b \, = \, \frac{u(2y_0)}{u_{\max}(2y_0)} \, = \, \frac{ (1 + a^{\frac{1}{\alpha-1}})^{\alpha -1} }{2^{\alpha-1}} \, > \, a.
$$
The proof is complete.
\end{proof}

A similar argument shows:

\begin{lemma}
\label{fail-lemma2}
For any $0 < a, \varepsilon, \rho < 1$, there exists an $0 < r < 1$ such that
\begin{equation}
\label{eq:fail-lemma}
\Lambda_{r}[a \cdot u_{\max}] > (1-\varepsilon) \cdot u_{\max}, \qquad \text{on }\mathbb{D}_\rho.
\end{equation}
 \end{lemma}

\subsection{Invisible measures}
\label{sec:invisible}

 Suppose $\mu$ is a measure on the unit circle that does not charge $\alpha$-Beurling-Carleson sets.
In order to show that $\mu$ is invisible, it is enough to check that $\Lambda[u_{\max} - P_\mu] = u_{\max}$, where $\Lambda$ denotes the minimal dominating solution on the unit disk. 

According to Corollary \ref{roberts-series}, for any parameters $c, j_0$, we can express $\mu$ as an infinite series
$$
\mu = \mu_1 + \mu_2 + \dots,
$$
where $\mu_j$ satisfies the modulus of continuity estimate 
\begin{equation}
\label{eq:roberts-mod}
\mu_j(I) \le c|I|^{\alpha}, \qquad I \in \mathcal D_{j + j_0}.
\end{equation}
One may express the condition (\ref{eq:roberts-mod}) in terms of the Poisson extension $P_{\mu_j}$ to the unit disk:
$$
P_{\mu_j}(z) \, \le \, c_2(1-|z|)^{\alpha-1} \, \le \, c_3 \cdot u_{\max}(z), \qquad  |z| = 1 - 2^{-(j+j_0)}.
$$
We choose the parameter $c > 0$ in the Roberts decomposition sufficiently small so that the above equation holds with $c_3 = b-a$, where $0 < a < 1$ is arbitrary and $b = b(a)$ is given by Lemma \ref{fail-lemma}.

By Lemma \ref{one-at-a-time} and monotonicity properties of $\Lambda$, we have
\begin{align*}
 \Lambda \bigl [ u_{\max} - P_{\mu} \bigr ] & = \lim_{j \to \infty} \Lambda \bigl [ u_{\max} - P_{\mu_1+\mu_2+\dots+\mu_j} \bigr ] \\
& =  \lim_{j \to \infty} \Lambda \bigl [\dots  \Lambda[u_{\max} - P_{\mu_j}] \dots - P_{\mu_1} \bigr  ]\\
& \ge  \lim_{j \to \infty}   \Lambda_{1-1/n_1} \bigl [\dots  \Lambda_{1-1/n_j}[u_{\max} - P_{\mu_j}] \dots - P_{\mu_1} \bigr  ].
\end{align*}
Since each time we apply $\Lambda_{1-1/n_i}$, we shrink the domain of the definition, the above inequality is valid on $\mathbb{D}_{1-1/n_1}.$
Using the restoring property $j$ times, we get
 $$
 \Lambda \bigl [ u_{\max} - P_{\mu} \bigr ] \ge a \cdot u_{\max}, \qquad \text{on }\mathbb{D}_{1-1/n_1}.
 $$
 Applying the restoring property one more time shows that for any given $0 < \rho < 1$ and $\varepsilon > 0$, one could choose the offset $j_0 \ge 0$ sufficiently large to guarantee that
  $$
 \Lambda \bigl [ u_{\max} - P_{\mu} \bigr ] \ge (1 - \varepsilon) u_{\max}, \qquad \text{on }\mathbb{D}_{\rho}.
 $$
 In other words, $ \Lambda \bigl [ u_{\max} - P_{\mu} \bigr ] = u_{\max}$ as desired.

\subsubsection{What happens when \texorpdfstring{$1 < p \le 3$}{1 < p \textle\ 3}?}

If $1 < p \le 3$, then by Harnack's inquality,
$$
P_{\mu}(z) \, \le \, 2(1-|z|)^{-1} \mu(\partial \mathbb{D}) \, \lesssim \, u_{\max}(z), \qquad |z| < 1,
$$ is true for any measure on the unit circle. By multiplying $\mu$ by a small constant $\varepsilon > 0$, one can arrange that
$P_{\varepsilon \mu} \le (1/2)  u_{\max}$ or $u_{\max} - P_{\varepsilon \mu} \ge (1/2)u_{\max}$. The argument above shows that 
$ \Lambda \bigl [ u_{\max} - P_{\mu} \bigr ] = u_{\max}$, which means that the measure $\varepsilon \mu$ is invisible. In turn, this implies that $\mu$ itself is invisible.

\subsection{Special Privalov Stars}
\label{sec:special-star}

 Suppose $E \subset \partial \mathbb{D}$ is a $\beta$-Beurling-Carleson set with $\beta < \alpha$ and $\mu$ is a measure supported on $E$. Given $\varepsilon > 0$, we will construct a special sawtooth domain $\tilde{K}_E = \tilde{K}_{E}(\varepsilon, \mu) \subset \mathbb{D}$ containing the origin which satisfies the following properties:
\begin{enumerate}
\item[(1)] Let $\omega_z$ denote the harmonic measure on $\partial \tilde{K}_E$ as viewed from $z \in \tilde{K}_E$. We require that
$$
\int_{\partial \tilde{K}_E} u_{\max}(z) d\omega_0(z) \, \asymp \, \int_{\partial \tilde{K}_E} (1-|z|)^{\alpha-1} d\omega_0(z) \, < \, \infty.
$$
\item[(2a)] Secondly, we want the Riemann map $\varphi: (\mathbb{D},0) \to (\partial \tilde{K}_E,0)$ to have a finite angular derivative at $\varphi^{-1}(\zeta)$ for $\mu$ a.e.~$\zeta \in E = \partial \tilde{K}_E \cap \partial \mathbb{D}$.

\item[(2b)]  In view of the Schwarz lemma, for any $\zeta \in E$, the angular derivative $1 < |\varphi'(\varphi^{-1}(\zeta))| < \infty$, or alternatively, 
$0 < |(\varphi^{-1})'(\zeta)| < 1$.
We will construct  $\partial \tilde{K}_E$ so that the set $E' \subset E$  where $1-\varepsilon < |(\varphi^{-1})'(\zeta)| < 1$ has measure $\mu(E') \ge (1-\varepsilon) \mu(E)$.
\end{enumerate} 

\noindent 
Fix a constant $1 < \gamma < \frac{1}{1-\alpha}$.
We fix a $C^1$ function $\phi: [0,1] \to [0,1]$ which satisfies

\smallskip

$0 < \phi(t) \le 1 - 2 |t-1/2|$ for $0 < t < 1$,

$\phi(0) = 0$, $\phi(t) \sim t^{\gamma}$ as $t \to 0$,

$\phi(1/2) = 1$,

$\phi(1) = 0$, $\phi(t) \sim (1-t)^{\gamma}$ as $t \to 1$,

\smallskip

\noindent 
and define the {\em tent} over  $[0,1]$ with height $h$ by
$$
T^h_{[0,1]} = \bigl  \{(x,y) \in \mathbb{R}^2 : 0 \le x \le 1, \, 0 \le y \le  h \cdot \phi(x) \bigr \}.
$$
Let $\{h(I)\} \subset (0,1]$ be a collection of heights. Over each complementary arc $I = (e^{i\theta_1}, e^{i\theta_2}) \subset \partial \mathbb{D} \setminus E$, we  build the tent 
$$
T_I \, = \, \biggl \{ r e^{i\theta} \, : \, \theta_1 \le \theta \le \theta_2, \ 1 - \psi \cdot h(I) \cdot \phi \biggl (\frac{\theta - \theta_1}{\theta_2 - \theta_1} \biggr ) \le r \le 1 \biggr \},
$$
where $0 < \psi \le 1$ is an auxiliary parameter to be chosen. The special Privalov star $\tilde{K}_E$ is then obtained by removing these tents from the unit disk.
To achieve the above objectives, we use the heights
\begin{equation}
\label{eq:choice-of-heights}
h(I) = \min \biggl (|I|, \, \frac{|I|^\alpha}{u(z_I)} \biggr ), \qquad u = P_\mu.
\end{equation}

\subsubsection{Condition (1)}

 For an arc $J \subset \partial \mathbb{D}$, we denote by $\tau(J)$   the part of $\partial \tilde{K}_E$ that is located above $J$ in $\partial \tilde{K}_E$, i.e.~$\tau(J) = \{z \in \partial \tilde{K}_E : z/|z| \in J \}$. %For brevity, we write $\omega(J) = \omega_{0}(\tau(J))$.

\begin{lemma}
\label{harmonic-and-lebesgue-measures}
The harmonic measure on $\partial \tilde{K}_E$ as viewed from the origin is bounded above by a multiple of arclength.
\end{lemma}

\begin{proof}
To prove the lemma, we show that $\omega_{\tilde {K}_E, 0}(\tau(J)) \lesssim |J|$ for any arc $J$ of the unit circle with $|J| \le 1/4$. Let $B = B(J)$ be a ball centered at the midpoint of $J$ of radius $3|J|$.
Since $\tau(J) \subset B(J)$, by the monotonicity properties of harmonic measure, we have
$$
\omega_{\tilde {K}_E, 0}(\tau(J)) \le \omega_{\mathbb{D} \setminus B, 0}(\partial B \cap \mathbb{D}).
$$
The latter quantity is easily seen to be $\lesssim |J|$.
\end{proof}

\begin{corollary}
For a complementary arc $I \subset \partial \mathbb{D} \setminus E$, we have
$$
\int_{\tau(I)} u_{\max}(z) d\omega_0(z)  \lesssim
 h(I)^{\alpha-1} \cdot |I|.
$$
\end{corollary}

\begin{proof}
We  split $I = \bigcup_{n \in \mathbb{Z}} I_n$ into countably many Whitney arcs, so that  $|I_n| = (1/2)^{|n|} \cdot |I_0|$  and $I_m$ and $I_n$ have a common endpoint if $|m-n|=1$. In view of the above lemma,
$$
\int_{\tau(I_n)} u_{\max}(z) d\omega_0(z)  \lesssim \frac{|I|}{2^{|n|}} \cdot \biggl \{ \frac{h(I)}{2^{\gamma |n|}}\biggr \}^{\alpha-1}.
$$
By the choice of $\gamma$, the corollary follows after summing a convergent geometric series.
\end{proof}

We now verify Condition (1). With the choice of heights (\ref{eq:choice-of-heights}),
\begin{equation*}
\int_{\partial \tilde{K}_E} u_{\max}(z) d\omega_0(z) \lesssim \sum |I|^{\alpha^2-\alpha+1} u(z_I)^{1 - \alpha}.
\end{equation*}
Applying H\"older's inequality with exponents $1/\lambda$  and $1/(1- \lambda)$, we get
\begin{align}
\sum |I|^{\alpha^2-\alpha+1} u(z_I)^{1 - \alpha} 
& = \sum   |I|^{\alpha(\alpha-1)+\lambda} \cdot  |I|^{1-\lambda}  u(z_I)^{1-\alpha} \nonumber \\
& \le \Bigl (\sum |I|^{\frac{\alpha(\alpha-1)+\lambda}{\lambda}} \Bigr )^{\lambda} \Bigl (\sum |I|  u(z_I)^{\frac{1-\alpha}{1-\lambda}} \Bigr )^{1 - \lambda}.
\label{eq:holder-abc}
\end{align}
With the choice
$$
\lambda \, = \, \alpha \cdot \frac{1-\alpha}{1-\beta} \, < \, \alpha,
$$
we have
$$
\beta = \frac{\alpha(\alpha-1)+\lambda}{\lambda}  \qquad \text{and} \qquad \delta \, = \, \frac{1-\alpha}{1-\lambda} \, < \, 1.
$$
The first sum in (\ref{eq:holder-abc}) is finite as $E$ is a $\beta$-Beurling-Carleson set, while the
second sum is finite since the non-tangential maximal function of $u$ lies in $L^\delta$.

%\begin{remark}
%The integral $ \int_{\partial \tilde{K}_E} P_\mu(z) |dz|$ with respect to arclength may be infinite if the maximal function of $P_\mu$ is not in $L^1$.
%\end{remark}

\subsubsection{Conditions (2a) and (2b)}

In order to verify that the special sawtooth domain $\tilde{K}_E$ satisfies Condition (2a), we need to check the Rodin-Warschawski condition for the existence of an angular derivative, cf. Theorem \ref{rodin-warschawski}. This will be done in Lemmas \ref{condition2a-1}
 and \ref{condition2a-2} below.

For a point $\zeta \in \partial \mathbb{D}$, we write $H(\zeta)$ for the length of the radius $[0, \zeta]$ that lies outside of $\tilde {K}_E$.

\begin{lemma}
\label{condition2a-1}
For a point $x \in E$ and a complementary arc $I \subset \partial \mathbb{D} \setminus E$, we have
$$
\int_{x e^{i \eta} \in I} \frac{H(xe^{i \eta})}{\eta^2} \, d\eta  \, \lesssim \, \frac{h(I) \cdot |I|}{\dist(x, I/2)^2}.
$$
\end{lemma}

\begin{proof}
 We decompose $I = \bigcup_{n \in \mathbb{Z}} I_n$ into a union of countably many Whitney arcs so that 
 $|I_n| = (1/2)^{|n|} \cdot |I_0|$ and $I_m$ and $I_n$ have a common endpoint if $|m-n|=1$.
Since $\dist(x, I_n) \ge 2^{-|n|} \dist(x, I/2)$,
\begin{align*}
\int_{x e^{i \eta} \in I_n} \frac{H(xe^{i \eta})}{\eta^2} \, d\eta & \lesssim \frac{\{\max_{\zeta \in I_n} H(\zeta)\} \cdot |I_n|}{\dist(x, I_n)^2} \\
& \lesssim  \frac{2^{-\gamma |n|} h(I) \cdot 2^{-|n|}|I|}{\{ 2^{-|n|} \dist(x, I/2) \}^2} \\
& =  2^{-(\gamma-1)|n|} \cdot \frac{h(I) \cdot |I|}{\dist(x, I/2)^2}.
\end{align*}
The lemma follows after summing a convergent geometric series.
\end{proof}

\begin{lemma}
\label{condition2a-2}
For $\mu$ a.e.~$x \in \partial \mathbb{D}$, the sum over complementary arcs
$$
\sum \frac{h(I) \cdot |I|}{\dist(x, I/2)^2} < \infty.
$$
\end{lemma}

\begin{proof}
It is enough to check that
\begin{align*}
\int_{\partial \mathbb{D}} \biggl \{ \sum_I \frac{h(I) \cdot |I|}{\dist(x, I/2)^2} \biggr \} \,d\mu(x)
& \le 
\int_{\partial \mathbb{D}} \biggl \{ \sum_I \frac{|I|^{\alpha+1}}{u(z_I) \cdot \dist(x, I/2)^2} \biggr \} \, d\mu(x) \\
& =  \sum_I |I|^\alpha \cdot  \biggl \{ \frac{1}{u(z_I)} \int_{\partial \mathbb{D}} \frac{|I|}{ \dist(x, I/2)^2} \, d\mu(x) \biggr \}
\end{align*}
is finite. To see this, notice that the expression in the parentheses is $O(1)$ and use that $E$ is a $\beta$-Beurling-Carleson set (and hence, an $\alpha$-Beurling-Carleson set).
\end{proof}

In view of Lemma \ref{increasing-angular-derivatives},  to achieve Condition (2b), we only need to select a sufficiently small auxiliary parameter $0 < \psi \le 1$.

\subsection{Existence}
\label{sec:existence}

To prove Theorem \ref{nearly-maximal}, it remains to construct a nearly-maximal solution with deficiency measure $\mu$ supported on a $\beta$-Beurling-Carleson set $E$.

For $n \in \mathbb{R}$, let $u_n$ be the solution of $\Delta u = u^p \cdot \chi_{u > 0}$ which is equal to $n$ on the unit circle.
Since $u_n - P_\mu$ is a subsolution and $n - P_\mu$ is a supersolution of $\Delta u = u^p \cdot \chi_{u > 0}$  with the same boundary data, by the principle of sub- and super-solutions, there exists a unique solution $u_{\mu, n}$ such that
\begin{equation}
\label{eq:u-mu-n-bound}
u_n - P_\mu \, \le \, u_{\mu, n} \, \le \, n - P_\mu.
\end{equation}
As the solutions $u_{\mu, n}$ are increasing in $n$ and bounded above by $u_{\max}$, the limit $u := \lim_{n \to \infty} u_{\mu, n}$ exists. Taking $n \to \infty$ in (\ref{eq:u-mu-n-bound}), we get
$$
u_{\max} - P_\mu \le u,
$$
which tells us that $u$ is a nearly-maximal solution whose deficiency measure is at most $\mu$. 

To show that the mass of the deficiency measure of $u$ is at least $\mu(\partial \mathbb{D})$, we use the special sawtooth domain $\tilde{K}_E$ constructed in Section \ref{sec:special-star}.
For $0 < r < 1$, we form the truncated region $K_r = \tilde{K}_E \cap \mathbb{D}_r$. Its boundary consists of two parts: a {\em sawtooth} part $\partial_{\saw} K_r  = \partial K_r  \setminus \partial \mathbb{D}_r$ and a {\em round} part $\partial_{\round} K_r = \partial K_r  \cap \partial \mathbb{D}_r$.
We estimate $u_{\mu, n}$ on $\partial K_r$ by
 $$
u_{\mu, n} \, \le \,  f \, := \,  \begin{cases} u_{\max}, \quad &\text{on }\partial_{\saw} K_{r}, \\ n - P_\mu, \quad &\text{on }\partial_{\round} K_r. \end{cases}
 $$
By the maximum principle, $u$ is bounded above on $K_r$ by the harmonic extension of these boundary values.
Taking $r \to 1$ while keeping $n$ fixed, we get
\begin{align}
\label{eq:special-bound}
u(z) & \le \int_{\partial \tilde{K}_E} u_{\max}(w) d\omega_z(w) - \lim_{r \to 1} \int_{\partial_{\round} K_r} P_\mu(w) d\omega_{K_r, z}(w) \\
& = A(z) - B(z),
\end{align}
for $z \in \tilde{K}_E$. In the equation above, $\omega_z$ and $\omega_{K_r,z}$ denote harmonic measures from the point $z$ in the domains $\tilde{K}_E$ and $K_r$ respectively.
Condition (1) guarantees that $A(z)$ is finite. Below, we will see that Condition (2) ensures that $B(z)$ is sufficiently large to be  responsible for the deficiency of $u$.

\subsubsection{A lemma featuring Privalov stars}

For a closed subset $F \subset \partial \mathbb{D}$, we write $K_{F, \theta}$ for the standard Privalov star, which is defined as the union of Stolz angles emanating from $F$ with aperture $0 < \theta < \pi$. We will use the following elementary lemma:

\begin{lemma}
\label{star-lemma}
Let $\mu$ be a positive measure on the unit circle and $F \subset \partial \mathbb{D}$ be a closed set. For any aperture $0 < \theta < \pi$,
$$
\limsup_{\rho \to 1} \int_{K_{F,\theta} \cap \partial \mathbb{D}_\rho} P_\mu(w) |dw| \le \mu(F).
$$
Conversely, for any $\varepsilon > 0$, there exists an aperture $0 < \theta < \pi$ so that
$$
\liminf_{\rho \to 1} \int_{K_{F,\theta} \cap \partial \mathbb{D}_\rho} P_\mu(w) |dw| \ge (1-\varepsilon) \mu(F).
$$
\end{lemma}

\subsubsection{Pruning the set $E$ further}

Recall that $E'$ was defined as the subset of $E$ where the angular derivative $1-\varepsilon < |(\varphi^{-1})'(\zeta)| < 1$, and we had arranged that $\mu(E') \ge (1 - \varepsilon) \mu(E)$. 
By sacrificing a little bit more mass, we can obtain uniformity of non-tangential limits and truncated Stolz angles.
More precisely, for any $\varepsilon > 0$ and $\theta > 0$, one can find a closed subset $E'' \subset E'$ and $0 < \rho_0 < 1$ such that
\begin{equation}
\label{eq:uniformity1}
\mu(E'') \ge (1- 2\varepsilon)\mu(E),
\end{equation}
\begin{equation}
\label{eq:uniformity2}
1- 2\varepsilon < |(\varphi^{-1})'(z)| < 1 + \varepsilon, \qquad \text{for }z \in K_{E'', \theta} \cap \{ \rho_0 < |w| < 1\}
\end{equation}
and
\begin{equation}
\label{eq:uniformity3}
K_{E'', \theta} \cap \{ \rho_0 < |w| < 1\} \subset \tilde{K}_E.
\end{equation}

\subsubsection{Strategy}

To prove the existence part of Theorem \ref{nearly-maximal}, we show:

\begin{lemma}
\label{nearly-maximal2}
For any $\varepsilon > 0$, we can choose the aperture $0 < \theta < \pi$ sufficiently close to $\pi$ so that
\begin{equation}
\label{eq:H-estimate}
\int_{K_{E'', \theta}\cap \partial \mathbb{D}_\rho} A(z) |dz| \le \varepsilon \cdot  \mu(E''),
\end{equation}
 and
\begin{equation}
\label{eq:P-estimate}
\int_{K_{E'', \theta} \cap \partial \mathbb{D}_\rho} B(z) |dz| \ge (1-\varepsilon) \cdot \mu(E''),
\end{equation}
for all $\rho_0 < \rho < 1$ sufficiently close to 1.
\end{lemma}
%To form the set $F_\rho$, we will slightly truncate the set $E$, i.e. pass to a closed subset $E'$ with $\mu(E') \ge (1- \varepsilon) \mu(E)$ and set $F_\rho = K_{E', \theta} \cap \partial \mathbb{D}_\rho$.

 \begin{proof}[Proof of existence in Theorem \ref{nearly-maximal} assuming Lemma \ref{nearly-maximal2}]
 Decompose $u = u_+ - u_-$ into positive and negative parts.
For $\rho_0 < \rho < 1$, we have
\begin{align*}
\int_{|z| = \rho} (u_{\max}(z) - u(z)) |dz| & \ge \int_{|z| = \rho} u_- (z) |dz| \\
& \ge  \int_{K_{E'',\theta} \cap \partial \mathbb{D}_\rho} \bigl ( B(z) - A(z) \bigr ) |dz| \\
& \ge  (1 - 2\varepsilon) \mu(E'') \\ 
& \ge  (1 - 2\varepsilon)^2 \mu(E).
\end{align*}
Since $\varepsilon > 0$ was arbitrary,  the mass of the deficiency measure of $u$ is at least $\mu(E)$.  
\end{proof}

The remainder of the paper is devoted to proving Lemma \ref{nearly-maximal2}.

\subsubsection{Estimating $A(z)$}
\label{sec:A-estimate}

Notice that $A(z)$ is a positive harmonic function on $\tilde{K}_E$ which extends absolutely continuous boundary values $u_{\max} \in L^1(\partial \tilde{K}_E, \omega_0)$. Therefore, if $\varphi$ is a conformal map from $(\mathbb{D},0)$ to $(\tilde{K}_E, 0)$, then $A \circ \varphi$ is a positive harmonic function on the unit disk with absolutely continuous boundary values on the unit circle. Since $\varphi^{-1}(E'')$ has Lebesgue measure zero by Loewner's lemma, Lemma \ref{star-lemma} tells us that
$$
\lim_{\rho \to 1} \int_{K_{\varphi^{-1}(E''),\theta} \cap \partial \mathbb{D}_\rho} (A \circ \varphi)(w) |dw| = 0.
$$
From here, (\ref{eq:H-estimate}) follows after an application of Harnack's inequality. 

\subsubsection{Estimating $B(z)$}

Since $\partial K_r = \partial_{\round} K_r \cup \partial_{\saw} K_r$,
\begin{equation*}
\int_{\partial_{\round} K_r} P_\mu(w) d\omega_{K_r, z}(w) = P_\mu(z) - \int_{\partial_{\saw} K_r} P_\mu(w) d\omega_{K_r, z}(w), \qquad z \in K_r.
\end{equation*}
By the monotonicity properties of harmonic measure, the integrals over $\partial_{\saw} K_r$ are increasing in $r$.
Taking $r \to 1$, we get
\begin{equation}
\label{eq:B-expression2}
B(z) = P_\mu(z) -\int_{\partial \tilde{K}_E \cap \mathbb{D}} P_\mu(w) d\omega_{z}(w), \qquad z \in \tilde{K}_E.
\end{equation}
%At first glance, the above equation may appear a little strange. The point is that $P_\mu$ is not the Poisson extension of its boundary values as it is singular on $E$.
Since $B$ is a positive harmonic function on $\tilde{K}_E$, the composition $B \circ \varphi$ is a positive harmonic function on the unit disk. Inspection shows that 
$B \circ \varphi = P_\nu$ for a positive measure $\nu$ supported on $\varphi^{-1}(E)$.
In fact, Theorem \ref{radon-nikodym2} tells us that 
$$
\nu = \varphi^* \bigl ( |\psi'(\zeta)| d\mu(\zeta) \bigr ).
$$
Since $1 - \varepsilon < |\psi'(\zeta)| < 1$ on $E' \supseteq E''$ by Condition (2b),
$$
\nu(\varphi^{-1}(E'')) \ge (1-\varepsilon)\mu(E'').
$$
Now, by Lemma \ref{star-lemma}, if the aperture $\theta$ is sufficiently close to $\pi$, then
\begin{align*}
 \liminf_{\rho \to 1} \int_{K_{\varphi^{-1}(E''),\theta}  \cap \partial \mathbb{D}_\rho} (B \circ \varphi)(w) |dw| 
& \ge   (1-\varepsilon) \nu(\varphi^{-1}(E'')) \\
 & \ge  (1-\varepsilon)^2 \mu(E'').
\end{align*}
The estimate (\ref{eq:P-estimate}) follows from Harnack's inequality as in Section \ref{sec:A-estimate}.

 \subsection*{Acknowledgements}
 
 The authors wish to thank Adem Limani for discovering a mistake in a previous version of the paper.
This research was supported by the Israeli Science Foundation (grant 3134/21), the Generalitat de Catalunya (grant 2021 SGR 00071), the Spanish Ministerio de Ciencia e Innovación (project PID2021-123151NB-I00) and the Spanish Research Agency (Mar\'ia de Maeztu Program CEX2020-001084-M).

\bibliographystyle{amsplain}

\begin{thebibliography}{00}

\bibitem[Ahe79]{ahern} P.~Ahern, {\em The Mean Modulus and the Derivative of an Inner Function}\/, Indiana University Mathematics Journal 28 (1979),  no. 2, 311--47.

\bibitem[AC74]{ahern-clark} P.~Ahern, D.~N.~Clark, {\em On inner functions with $B^p$ derivative}\/, Michigan Mathematical Journal 21 (1974), no. 2, 115--127.

\bibitem[BM92]{BM}C.~Bandle, M.~Marcus, {\em `Large' solutions of semilinear elliptic equations: Existence, uniqueness and asymptotic behaviour}\/, J. Anal. Math. 58 (1992), 9--24.

\bibitem[BM98]{BM2} C.~Bandle, M.~Marcus, {\em On second order effects in the boundary behaviour of large solutions of semilinear elliptic problems}\/, Differential Integral Equations 11 (1998) 23--34.

\bibitem[BM04]{BM3} C.~Bandle, M.~Marcus, {\em Dependence of blowup rate of large solutions of semilinear elliptic equations on the curvature of the boundary}\/, Complex Var. Theory Appl. 49 (2004), no 7--9, 555--570.

\bibitem[BBC87]{BBC} R.~D.~Berman, L.~Brown, W.~S.~Cohn, {\em Moduli of continuity and generalized BCH sets}\/,
Rocky Mt. J. Math. 17 (1987), no. 2, 315--338

 \bibitem[BK22]{BK} D.~Betsakos, N.~Karamanlis, {\em Conformal invariants and the angular derivative problem}\/, J. London Math. Soc. 105 (2022), no. 1, 587--620.
 
\bibitem[Beu40]{beurling} A.~Beurling, {\em Ensembles exceptionnels}\/, Acta. Math. 72 (1940), 1--13.
 
\bibitem[Bor13]{borichev} A.~Borichev, {\em Generalized Carleson-Newman inner functions}\/, Math. Z. 275 (2013), 1197--1206.
 
\bibitem[BNT17]{BNT}A.~Borichev, A.~Nicolau, P.~J.~Thomas, {\em Weak embedding property, inner functions and entropy}\/, Math. Ann. 368 (2017), 987--1015.

\bibitem[Bur86]{burdzy}K.~Burdzy, {\em Brownian excursions and minimal thinness. III. Applications to the angular derivative problem}\/, Math. Z. 192 (1986), no. 1, 89--107.

\bibitem[Car52]{carleson} L.~Carleson, {\em Sets of uniqueness for functions regular  in the unit circle}\/, Acta. Math. 87 (1952), 325--345.

\bibitem[Cul71]{cullen}M.~Cullen, {\em Derivatives of singular inner functions}\/, Michigan Math. J. 18 (1971), no. 3, 283-287.

\bibitem[DL02]{DL} M.~Del Pino, R.~Letelier, {\em The influence of domain geometry in boundary blow-up elliptic problems}\/, Nonlinear Anal. 48 (2002), no. 6, 897--904.

\bibitem[DK06]{DK} K.~Dyakonov, D.~Khavinson, {\em Smooth functions in star-invariant subspaces}\/, Recent advances in operator-related function theory, Contemp. Math. 393, Amer. Math. Soc., Providence, RI, 2006, 59–66. 

\bibitem[Gar09]{garcia-melian}J.~García-Melián, {\em Uniqueness of positive solutions for a boundary blow-up problem}\/, J. Math. Anal. Appl. 360 (2009), 530--536.

\bibitem[GM05]{garnett-marshall} J.~B.~Garnett, D.~E.~Marshall, {\em Harmonic  Measure}\/, New Mathematical Monographs 2, Cambridge University Press, 2005.

\bibitem[GMN08]{GMN}P.~Gorkin, R.~Mortini, N.~Nikolski, {\em Norm controlled inversions and a corona theorem for $H^\infty$-quotient algebras}\/, J. Funct. Anal. 255 (2008), 854--876.

\bibitem[GN16]{grohn-nicolau} J.~Gr\"ohn, A.~Nicolau, {\em Inner Functions in certain Hardy-Sobolev Spaces}\/, 
J. Funct. Anal. 272 (2017), no. 6, 2463--2486.

\bibitem[HJ94]{HJ} V.~Havin and B.~J\"oricke, {\em The Uncertainty Principle in Harmonic Analysis}\/, Ergebnisse der Mathematik, Springer-Verlag, Berlin, Heidelberg, 1994.

\bibitem[Ivr19]{ivr19} O.~Ivrii, {\em Prescribing inner parts of derivatives of inner functions}, J. d'Analyse Math. 139 (2019), 495--519.

\bibitem[IK22]{IK22} O.~Ivrii, U.~Kreitner, {\em Critical values of inner functions}, 2022. arXiv:2212.14818.

\bibitem[Kel57]{keller} J.~B.~Keller, {\em On solutions of $\Delta u = f(u)$}\/, Comm. Pure Appl. Math. 10 (1957), 503--510.

\bibitem[Kor81]{korenblum-cyc} B.~Korenblum, {\em Cyclic elements in some spaces of analytic functions}\/, Bull. Amer. Math. Soc. 5 (1981), 317--318.

\bibitem[Kor93]{korenblum}B.~Korenblum, {\em Outer Functions and Cyclic Elements in Bergman Spaces}\/, J. Funct. Anal. 151 (1993), no. 1,  104--118.

\bibitem[LaM94]{LM} A.~C.~Lazer, P.~J.~McKenna, {\em Asymptotic behaviour of solutions of boundary blow-up problems}\/, Differential Integral Equations 7 (1994) 1001--1019.

\bibitem[LiM22a]{LM22} A.~Limani, B.~Malman, {\em On model spaces and density of functions smooth on the boundary}\/, Rev. Mat. Iberoam. 2022, to appear.

\bibitem[LiM22b]{LM21b}  A.~Limani, B.~Malman, {\em Constructions of some families of smooth Cauchy transforms}\/, Can. J. Math. 2022, to appear.

\bibitem[LiM23]{LM21a}  A.~Limani, B.~Malman, {\em On the problem of smooth approximations in de Branges-Rovnyak spaces and connections to subnormal operators}\/, J. Funct. Anal. 284 (2023), no. 5, 109803.

\bibitem[Mak89]{makarov}N.G.~Makarov, {\em On a class of exceptional sets in the theory of conformal mappings}\/, Mat. Sb. 9 (1989) 1171--1182. In Russian.

\bibitem[Mal22]{Mal22} B.~Malman, {\em Thomson decompositions of measures in the disk}\/, 2022. arXiv:2208.08810.

\bibitem[Mas12]{mashreghi}J.~Mashreghi, {\em Derivatives of Inner Functions}\/, Fields Institute Monographs, 2012.

\bibitem[Oss57]{osserman} R.~Osserman, {\em On the inequality $\Delta u\geq f (u)$}\/, Pac. J. Math. 7 (1957), no. 4, 1641--1647.


\bibitem[Pom76a]{pommerenke}C.~Pommerenke, {\em On the Green's function of Fuchsian groups}\/, Ann. Acad. Sci. Fenn. Math. 2 (1976), 409--427.

\bibitem[Pom76b]{pommerenke2}C.~Pommerenke, {\em On automorphic forms and Carleson sets}\/, Michigan Math J. 23 (1976), no. 2, 129--136.

\bibitem[Pon16]{ponce-book}A.~C.~Ponce, {\em Elliptic PDEs, Measures and Capacities: From the Poisson Equation to Nonlinear Thomas--Fermi Problems}\/, EMS Tracts in Mathematics 23, 2016.

\bibitem[RS18]{reijonen-sugawa} A.~Reijonen, T.~Sugawa, {\em Characterizations for inner functions in certain function spaces}\/, 
Complex Anal. Oper. Theory 13 (2019), 1853--1871.

\bibitem[Rob85]{roberts} J.~W.~Roberts, {\em Cyclic inner functions in the Bergman spaces and weak outer functions in $H^p$,
$0 < p < 1$}\/, Illinois J. Math. 29 (1985), 25--38.

\bibitem[Sha80]{shapiro} J.~H.~Shapiro, {\em Hausdorff measure and Carleson thin sets}\/, Proc. Amer. Math. Soc. 9 (1980), no. 1, 67--71.

\bibitem[TW70]{taylor-williams} B.~A.~Taylor, D.~L.~Williams, {\em Ideals in rings of analytic functions with smooth boundary values}\/, Canad. J. Math. 22 (1970), 1266--1283.

\bibitem[Wid71]{widom} H.~Widom, {\em $H_p$ sections of vector bundles over Riemann surfaces}\/, Ann. of Math. 94 (1971), 304--324.

\end{thebibliography}

\end{document}